\documentclass[a4paper,11pt,pdf]{amsart}
\usepackage{enumerate, amsmath, amsfonts, amssymb, amsthm,  wasysym, graphics, graphicx, xcolor, frcursive,bbm}

\usepackage{hyperref}
\usepackage[all]{xy}
\usepackage[T1]{fontenc}

\usepackage{setspace}

\usepackage[scr=boondoxo,scrscaled=1.05]{mathalfa}

\usepackage{etex}

\definecolor{darkblue}{rgb}{0.0,0,0.7} % darkblue color
\newcommand{\darkblue}{\color{darkblue}} % darkblue command
\definecolor{darkred}{rgb}{0.7,0,0} % darkred color
\newcommand{\defn}[1]{\emph{\darkblue #1}} % emphasis of a definition

\usepackage[all]{xy}
\usepackage[T1]{fontenc}

\usepackage[colorinlistoftodos]{todonotes}

\def\lA{\ell_A}
\def\FA{\mathcal{F}_{\Delta}^A}
\def\Bic{\mathscr{B}(\Phi^+)}

\def\W{{\mathcal{W}}}
\def\S{{\mathcal{S}}}
\def\H{{\mathcal{H}}}

\newcommand\bs{{\boldsymbol s}}
\newcommand\br{{\boldsymbol r}}
\newcommand\bt{{\boldsymbol t}}
\newcommand\bw{{\boldsymbol w}}
\newcommand\bx{{\boldsymbol x}}
\newcommand\bu{{\boldsymbol u}}
\newcommand\bv{{\boldsymbol v}}
\newcommand\by{{\boldsymbol y}}
\newcommand\bz{{\boldsymbol z}}

\def\T{{\mathcal{T}}}

\def\H{\mathcal{H}}
\def\Kb{K^b(\mathcal{B})}
\def\Kr{K^b(\mathcal{R})}
\def\Br{B_{\W}}

\newtheorem{theorem}{Theorem}[section]
\newtheorem{proposition}[theorem]{Proposition}
\newtheorem{lemma}[theorem]{Lemma}
\newtheorem{corollary}[theorem]{Corollary}
\newtheorem{conjecture}[theorem]{Conjecture}

\theoremstyle{definition}
\newtheorem{definition}[theorem]{Definition}
\newtheorem{remark}[theorem]{Remark}
\newtheorem{example}[theorem]{Example}

\usepackage{graphicx}                  
\usepackage{setspace}
\usepackage{eqnarray}

\newcommand{\ts}{\textsuperscript}

\def\W{{\mathcal{W}}}
\def\S{{\mathcal{S}}}
\def\B{{\mathcal{B}}}
\def\H{{\mathcal{H}}}

\def\T{{\mathcal{T}}}

\def\gi{\Gamma_{\geq i}^A}
\def\gii{\Gamma_{\geq i+1}^A}
\def\Kbp{K^b(\mathcal{B})^{\geq 0}}
\def\Kbn{K^b(\mathcal{B})^{\leq 0}}

\author{Thomas Gobet}
\address{TU Kaiserslautern, Fachbereich Mathematik, Postfach 3049, 67653 Kaiserslautern, Germany.}
\email{gobet@mathematik.uni-kl.de} 

\title[Twisted filtrations of Soergel bimodules]{Twisted filtrations of Soergel bimodules and linear Rouquier complexes}

\begin{document}

\maketitle

\begin{abstract}
{We consider twisted standard filtrations of Soergel bimodules associated to arbitrary Coxeter groups and show that the graded multiplicities in these filtrations can be interpreted as structure constants in the Hecke algebra. This corresponds to the positivity of the polynomials occurring when expressing an element of the canonical basis in a generalized standard basis twisted by a biclosed set of roots in the sense of Dyer, and comes as a corollary of Soergel's conjecture. We then show the positivity of the corresponding inverse polynomials in case the biclosed set is an inversion set of an element or its complement by generalizing a result of Elias and Williamson on the linearity of the Rouquier complexes associated to lifts of these basis elements in the Artin-Tits group. These lifts turn out to be generalizations of Mikado braids as introduced in a joint work with Digne. This second positivity property generalizes a result of Dyer and Lehrer from finite to arbitrary Coxeter groups.}
\end{abstract}

\tableofcontents

%%%%%%%%%%%%%%%%%%%%%%%%%%%%%%%%%%%%%%%%%%
%BRUHAT LIKE ORDERS
%%%%%%%%%%%%%%%%%%%%%%%%%%%%%%%%%%%%%%%%%%

\section{Introduction}

Let $(\W,\S)$ be a Coxeter system with $|\S|<\infty$. Let $\H$ be the corresponding Iwahori-Hecke algebra over the ring $\mathbb{Z}[v,v^{-1}]$ with standard basis $\{T_w\}_{w\in\W}$ and costandard basis $\{T_{w^{-1}}^{-1}\}_{w\in\W}$. Denote by $\T$ the set of conjugates of the elements of $\S$. In their seminal paper of $1979$ \cite{KL}, Kazhdan and Lusztig introduced two canonical bases $\{C_w\}_{w\in\W}$ and $\{C_w'\}_{w\in\W}$ of $\H$ and related them to the representation theory of $\H$ and $\W$. In case $\W$ is a finite Weyl group, the canonical bases are closely related to the geometry of Schubert varieties. Kazhdan and Lusztig conjectured that when writing an element $C_w'$ as a linear combination of elements of the standard basis, the coefficients are polynomials with nonnegative coefficients. These polynomials became known as \defn{Kazhdan-Lusztig polynomials} and are broadly studied in Lie theory, representation theory and combinatorics (see for instance \cite{BjBr} or \cite{Humph} for introductions to the topic). 

While Kazhdan and Lusztig proved their positivity conjecture in 1980 in case $\W$ is a  finite or affine Weyl group \cite{KLproof} using geometric methods, the general case remained mysterious until recently. Soergel proposed \cite{S1}, \cite{S} an approach allowing one to replace the geometry involved in the Weyl group case by a remarkable additive monoidal Krull-Schmidt category $\mathcal{B}$ of graded bimodules over a polynomial ring. These bimodules, nowadays called \defn{Soergel bimodules}, can be defined for an arbitrary Coxeter system and provide a categorification of (the canonical basis  $\{C_w'\}_{w\in\W}$ of) $\H$. In this framework, Soergel formulated a purely algebraic conjecture implying Kazhdan-Lusztig's positivity conjecture in full generality \cite{S} and proved it for finite Weyl groups, using again geometry but suggesting the existence of an algebraic proof. Soergel's conjecture was proven by Elias and Williamson in \cite{EW}. See also \cite{Willkl} and the references therein for an overview of the topic. 

More precisely, indecomposable Soergel bimodules are indexed (up to graduation shifts and isomorphism) by the elements of $\W$. Denote by $\{B_w\}_{w\in\W}$ the family of (unshifted) indecomposable Soergel bimodules up to isomorphism. Soergel showed \cite{S} that the split Grothendieck ring $\left\langle \mathcal{B}\right\rangle$ of his category is isomorphic to the Hecke algebra and conjectured that the class $\left\langle B_w\right\rangle$ of the bimodule $B_w$ corresponds exactly to the element $C_w'$ under this isomorphism. He explicitly described the isomorphism and its inverse, called the \defn{character} map. More precisely, the class $\left\langle B\right\rangle$ of a bimodule $B\in\mathcal{B}$ corresponds to the element of $\H$ given by
$$\left\langle B\right\rangle\mapsto \sum_{x\in \W} \sum_{i\mathbb{Z}} [B:\Delta_x(i)] v^{i+\ell(x)} T_x,$$
where $[B:\Delta_x(i)]$ denotes the multiplicity of a graded bimodule $\Delta_x$ (not lying in $\mathcal{B}$ if $x\neq e$) in some canonical filtration of $B$ called \defn{standard filtration}. Such a filtration requires one to fix a total order on the group which refines the Bruhat order. As a consequence, since the coefficients of the polynomials occurring on the right hand side are multiplicities, they are nonnegative. Hence Soergel's conjecture implies Kazhdan-Lusztig's positivity conjecture. 

As mentioned, the definition of the above filtration requires one to have fixed a linear extension of the Bruhat order on $\W$. Reversing the Bruhat order, Soergel shows that one also has a second filtration which can be thought of as a dual version of the standard filtration and can be related to structure constants in the costandard basis. More precisely, the character map can also be described via
$$\left\langle B\right\rangle\mapsto\sum_{x\in\W}{\sum_{i\in\mathbb{Z}} [B:\nabla_x (i)] v^{i-\ell(x)}} T_{x^{-1}}^{-1},$$
where $[B:\nabla_x (i)]$ denotes the multiplicity of a graded bimodule $\nabla_x$ in this second canonical filtration of $B\in\mathcal{B}$, called the \defn{costandard filtration} of $B$. 

Dyer considered \cite{Dyerth} the following positivity statements:

\begin{equation}\label{p1}
T_{x}^{-1} T_y\in\sum_{w\in\W} \mathbb{Z}_{\geq 0}[v^{\pm 1}] C_w,~\forall x,y\in\W,
\end{equation}

\begin{equation}\label{p2}
C_x' T_y\in\sum_{w\in\W} \mathbb{Z}_{\geq 0}[v^{\pm 1}] T_w,~\forall x,y\in\W.
\end{equation}

He gave a combinatorial proof that, for finite Coxeter groups, the conditions (\ref{p1}) and (\ref{p2}) are equivalent and proved both for universal Coxeter systems. Dyer and Lehrer \cite{DL} then proved condition (\ref{p2}) for finite Weyl groups using the fact that Kazhdan-Lusztig's positivity conjecture (i.e. (\ref{p2}) for $y=e$) is known for these groups and a geometrical argument in categories of perverse sheaves. Dyer later showed \cite{Dyerrep} (relying on partially unpublished results) that (\ref{p2}) would hold for finite Coxeter groups as a consequence of Soergel's conjecture. Grojnowski and Haiman \cite{GH} generalized (\ref{p2}) to affine Weyl groups also using geometric techniques. 

The aim of this paper is to provide general proofs of these two statements, without the finiteness assumption. Thanks to work of Elias and Williamson \cite{EW}, (\ref{p2}) for $y=e$ is known for arbitrary Coxeter systems, hence for proving (\ref{p2}) it remains to understand how one can find a replacement for Dyer and Lehrer's argument in the framework of Soergel's approach. Rewriting (\ref{p2}) as 

\begin{equation}\label{p3}
C_x' \in\sum_{w\in\W} \mathbb{Z}_{\geq 0}[v^{\pm 1}] T_w T_y^{-1},~\forall x,y\in\W,
\end{equation}
it is tempting to try to generalize Soergel's standard and costandard filtrations. Indeed, for $y=e$ above one has the expansion of $C_w'$ in the standard basis, while in case the group is finite one has the expansion in the costandard basis for $y=w_0$, the longest element in $\W$. Since we want to consider arbitrary Coxeter systems, this shows that considering bases of the form $\{T_w T_y^{-1}\}_{w\in\W}$ for various $y$ may not be sufficient since the costandard basis is in general not caught.  

In case $\W$ is finite, elements of the form $T_w T_y^{-1}$ appearing in the above paragraph or of the form $T_{x}^{-1} T_y$ as appearing in (\ref{p1}) turn out to be images of so-called \defn{Mikado braids} as introduced in a joint work with Digne \cite{DG}. In finite type, denoting by $\bw$ the canonical lift of $w$ in the Artin-Tits group $\Br$ attached to the Coxeter system, these braids can be defined as the elements which can be written in the form $\bu^{-1} \bv$ for $\bu$, $\bv$ canonical lifts of elements $u,v\in\W$ (we call such an element a \defn{right Mikado braid}), and these also turn out to be exactly the braids which can be written in the form $\bw \by^{-1}$ for some $w,y\in\W$ (we call such an element a \defn{left Mikado braid}). Hence in finite type, for any $w,y\in\W$, there exist $u,v\in\W$ such that $T_w T_y^{-1}=T_u^{-1} T_v$. Also, one can think of (\ref{p1}) as of a positivity of generalized inverse Kazhdan-Lusztig polynomials. 

For infinite groups however, it is false in general that left Mikado braids and right Mikado braids coincide. Nevertheless, Dyer proposed in \cite{Dyernil} a remarkable general definition which yields exactly the left (or right) Mikado braids for finite $\W$ but which strictly contain both sets in general. Dyer's definition relies on the theory of biclosed sets of roots or reflections (\cite{Dyerweak}), that is, sets of positive roots which are closed under positive real linear combinations, and whose complement is also closed. 

More precisely, to any biclosed set $A\subseteq\T$ and any $x\in\W$, Dyer associated \cite{Dyernil} a canonical lift $x_A$ of $x$ in $\Br$ obtained by lifting any reduced expression $s_1\cdots s_k$ of $x$ in the braid represented by the word $s_1^{\varepsilon_1}\cdots s_k^{\varepsilon_k}$, where $\varepsilon_i=-1$ if $s_k\cdots s_i\cdots s_k$ lies in $A$ and $\varepsilon_i=1$ otherwise. He showed that $x_A$ is independent of the chosen reduced expression for $x$ and that one can pass from any lifted reduced expression to any other by applying \defn{mixed braid relations}, that is, braid relations possibly involving inverses of the Artin generators. 

Denoting by $T_{x,A}$ the image of $x_A$ in the Hecke algebra, the set $\{T_{x,A}\}_{x\in\W}$ yields a basis of $\H$. For $A=\emptyset$ it is the standard basis, while for $A=\T$ it is the costandard basis. Dyer \cite{Dyershellings2} and Edgar \cite{Edgar} attached to any biclosed set $A$ a Bruhat-like order $<_A$ and a $\mathbb{Z}$-valued length function $\ell_A$. We use these orders to generalize Soergel's standard filtration and interpret structure constants of the canonical basis in the basis $\{T_{x,A}\}$ as graded multiplicities (Theorem \ref{threeparam}):

\begin{theorem}[Three-parameter Kazhdan-Lusztig positivity]\label{thm:3par}
Let $w\in \W$, $A\subseteq\T$ be biclosed. Write $C'_w=\sum_{x\in\W} h_{x,w}^A T_{x,A}$. Then $${h_{x,w}^A}=\sum_{i\in\mathbb{Z}} [B_w:\Delta_x^A(i)]_A v^{i+\ell_A(x)}.$$ 
In particular the generalized Kazhdan-Lusztig polynomials $h_{x,w}^A\in\mathbb{Z}[v,v^{-1}]$ have nonnegative coefficients.
\end{theorem}
In the above theorem, $[B:\Delta_x^A(i)]_A$ denotes the multiplicity of a graded bimodule $\Delta_x^A$ in some canonical filtration of $B$ analogous to Soergel's standard filtration, but defined using the "twisted" order $<_A$. In case the biclosed set $A$ is given by an inversion set $N(y)$ of an element $y\in\W$, one then has $T_{xy^{-1},N(y)}=T_x T_y^{-1}$, hence the above Theorem implies condition (\ref{p3}). In fact, a finite set $A$ is biclosed if and only if $A=N(y)$ for some element $y\in\W$, hence in case $\W$ is finite the positivity statement 
\begin{equation}
C'_w\in\sum_{x\in\W} \mathbb{Z}_{\geq 0}[v^{\pm 1}] T_{x,A},~\forall w\in\W, ~\forall A~\text{biclosed},
\end{equation}
implied by Theorem \ref{thm:3par} is exactly equivalent to (\ref{p3}).

 Notice that this approach is very similar to (generalizations of) that of \cite{Dyerrep} but that Dyer's approach might be more adapted for further generalizations involving biclosed sets.  

Note that for infinite groups, there are in general biclosed sets which are neither inversion sets of elements or their complements, hence the positivity statement that we get is more general than (\ref{p3}) in that case. 

The analogue of the positivity of inverse generalized Kazhdan-Lusztig polynomials in that setting would be
\begin{equation}\label{p5}
T_{x,A}\in\sum_{w\in\W} \mathbb{Z}_{\geq 0}[v^{\pm 1}] C_w,~\forall x\in\W, ~\forall A~\text{biclosed},
\end{equation}
which Dyer conjectures \cite{Dyer:private}. We are not able to prove (\ref{p5}) in such a level of generality, but we prove it in case $A=N(y)$ or $A=\T\backslash N(y)$ for $y\in\W$ (Theorem \ref{thm:inversepos}):
\begin{theorem}
Let $x,y\in \W$. Then 
$$T_{x}^{-1} T_y\in\sum_{w\in\W} \mathbb{Z}_{\geq 0}[v^{\pm 1}] C_w.$$ 
Similarly, $T_x T_y^{-1}\in\sum_{w\in\W} \mathbb{Z}_{\geq 0}[v^{\pm 1}] C_w.$
\end{theorem}
This corresponds to the case where $T_{x,A}$ comes from a left or right Mikado braid. In particular, (\ref{p1}) and the analogue version for left Mikado braids hold for arbitrary Coxeter groups. 

The strategy for the proof of the above theorem reproduces that of Elias and Williamson to prove the positivity of inverse Kazhdan Lusztig polynomials \cite{EW}. More precisely, we prove that any minimal Rouquier complex (i.e., a complex associated to a left or right Mikado braid in the bounded homotopy category of Soergel bimodules with no contractible direct summand) is "linear", that is, that any direct summand in homological degree $i$ of a miminal complex is isomorphic to $B_w(i)$ for some $w\in\W$ (Theorem \ref{thm:linearity}). Moreover, proving by induction that some $B_w$ cannot appear in certain degrees depending on parity conditions, we can conclude that (\ref{p5}) holds for left and right Mikado braids (Theorem \ref{thm:inversepos}). The key steps in the proof of Elias and Williamson are to use an explicit formula of Libedinsky-Williamson \cite{LW} which gives the standard filtration of a Rouquier complex of a positive permutation braid and then to use properties of the standard filtration to control the shifts. Libedinsky-Williamson's formula generalizes to the level of arbitrary Mikado braids (Proposition \ref{prop:libwill}; Libedinsky-Williamson's proof adapts mutatis mutandis), but unfortunately the used properties of the standard filtrations fail in general. The linearity of minimal complexes associated to left and right Mikado braids (Theorem \ref{thm:linearity}) is proven by combining Elias and Williamson's linearity statement and a dual version of it for negative canonical lifts of element of the Coxeter group in the Artin-Tits group.\\ 
~\\
{\bf Acknowledgements}. I thank Geordie Williamson and Wolfgang Soergel for discussions which initiated this work during the Darstellungstheorie Schwerpunkttagung in Bad Honnef in March 2015; the stay in Bad Honnef was funded by the DFG Schwerpunktprogramm Darstellungstheorie 1388 which I also thank. Special thanks go to Matthew Dyer for discussions and comments on a preliminary version, for making me aware of Corollary \ref{cor:doubletwist} and for kindly detailing me some generalizations of his results of \cite{Dyerrep}.

\section{Bruhat-like orders on a Coxeter group}

In this section, we introduce the Coxeter group terminology and twisted Bruhat orders on Coxeter groups. 
\subsection{Coxeter groups terminology}
Let $(\W,\S)$ be a Coxeter system with $\S$ finite. We denote by $\T$ the set of \defn{reflections} of $\W$, that is, the set $\bigcup_{w\in\W} w\S w^{-1}$. Let $\ell:\W\rightarrow\mathbb{Z}_{\geq 0}$ be the length function with respect to $\S$. Let $\leq$ denote the Bruhat order on $\W$. Recall its definition: for $u,v\in\W$, one has $u\leq v$ if and only if there exists a  sequence $u=u_0, u_1,\dots, u_k=v$, $u_i\in\W$ such that for any $i=0,\dots, k-1$, $\ell(u_{i+1})>\ell(u_i)$ and $u_{i+1}=u_i t_i$ for some $t_i\in\T$. For more on the general theory of Coxeter groups we refer the reader to \cite{bourbaki}, \cite{Humph}. 

The power set $\mathscr{P}(\T)$ may be seen as an abelian group under symmetric difference which we denote by $+$. Dyer characterized Coxeter groups as follows \cite[Lemmas 1.2 and 1.3]{Dyerth}:

\begin{proposition}[Dyer]\label{prop:dyercaract}
Let $\W$ be a group generated by a set $\S$ of involutions and $\T$ be the set of conjugates of elements of $\S$. Then the following are equivalent
\begin{enumerate}
\item $(\W,\S)$ is a Coxeter system,
\item There exists a function $N:\W\longrightarrow\mathscr{P}(\T)$ satisfying
\begin{enumerate}
\item $N(s)=\{s\}$ for any $s\in\S$,
\item $N(xy)=N(x)+x N(y) x^{-1}$ for all $x,y\in\W$.
\end{enumerate}
\end{enumerate}
Moreover, if such a function $N$ exists, then it is unique and for any $w\in\W$ we have $$N(w)=\{t\in\T~|~\ell(tw)<\ell(w)\}.$$
\end{proposition}
Starting from a Coxeter system $(\W,\S)$ and $w\in\W$ the set $N(w)$ as defined in the above Proposition is the \defn{(left) inversion set} of $w$. One has $|N(w)|=\ell(w)$.

\subsection{Twisted Bruhat orders}\label{sub:twisted}
We recall from \cite{Dyershellings2}, \cite{Dyer2} some facts about the twisted Bruhat orders on $\W$. Let $y\in\W$. One defines the \defn{$y$-twisted Bruhat order} $\leq_y$ by $$u\leq_y v~~\Leftrightarrow ~~uy\leq vy.$$
There are several equivalent characterizations of such an order, see [\cite{Dyer2}, 1.11]. For instance, it can be characterized as follows: define a $\mathbb{Z}$-valued length function $\ell_{y}:\W\rightarrow\mathbb{Z}$ by $$\ell_{y}(w)=\ell(w)-2|N(w^{-1})\cap N(y)|.$$
Then by \cite[1.11-1.12]{Dyer2} we have $u\leq_y v$ if and only if there exists a sequence $u=u_0, u_1,\dots, u_k=v$, $u_i\in\W$ such that for any $i=0,\dots, k-1$, $\ell_y(u_{i+1})>\ell_y(u_i)$ and $u_{i+1}=u_i t_i$ for some $t_i\in\T$.

Note that $\ell_y(e)=0$ for any $y\in\W$. The order $\leq_y$ has a unique minimal element given by $y^{-1}$. 

\begin{remark}\label{remark:reversed}
If $\W$ is finite and $u,v\in\W$, then $u \leq_{w_0} v$ if and only if $v\leq u$ (see \cite[Proposition 2.3.4]{BjBr}). Hence $\leq_{w_0}$ is the reverse Bruhat order. In that case $\ell_{w_0}(u)=-\ell(u)$ for any $u\in\W$. In case $\W$ is not finite, the reverse Bruhat order is not caught by a twisted order. This is one of the reasons for introducing a more general version of these orders in the next subsection. 
\end{remark}

\subsection{Order attached to a biclosed set of roots}
Dyer introduced more general versions of the twisted Bruhat orders as defined in Subsection \ref{sub:twisted} (see for instance \cite[Section 1]{Dyershellings2}). Given a Coxeter system $(\W,\S)$, denote by $V$ its reflection representation over $\mathbb{R}$ and by $\Phi$ the associated root system with positive system $\Phi^+\subseteq\Phi$. 

A subset $A\subseteq \T$ is \defn{closed} if the set $\Phi_A\subseteq \Phi^+$ of roots corresponding to $A$ has the following property: given $\alpha, \beta\in\Phi_A$, then any positive root lying in $\mathbb{R}_{>0}\alpha+\mathbb{R}_{>0}\beta$ lies again in $\Phi_A$. A subset $A\subseteq T$ is \defn{biclosed} if $A$ and $\T\backslash A$ are closed. Using \cite[Remark 3.2]{DyerBru}, this is equivalent to saying that for any dihedral reflection subgroup $\W'\subseteq \W$, the set $A\cap \W'$ is an initial section of a reflection order (we refer to \cite{Dyershellings2} and the references therein for definitions). We denote by $\mathscr{B}(\Phi^+)$ the set of biclosed sets of roots.

The following definition is given by Dyer \cite{Dyershellings2} for sets $A\subseteq\T$ which are initial sections of reflection orders and was extended to biclosed sets by Edgar \cite{Edgar}. Initial sections of reflection orders turn out to be biclosed (see \cite[2.1]{Dyerweak}) and Dyer furthermore conjectures that these should be exactly the biclosed sets (see \cite[2.2]{Dyerweak}). 

To any set $A\subseteq \T$ of reflections, associate a preorder $\leq_A$ on $\W$ as follows: $u\leq_A v$ if and only if there exist $t_1,\dots, t_n\in T$ with $v=u t_1 \cdots t_n$ such that $t_i\notin (N((u t_1\cdots t_{i-1}^{-1})+A)$ for all $i=1,\dots, n$. One has as in Subsection \ref{sub:twisted} a corresponding length function $\ell_A:\W\rightarrow\mathbb{Z}$ defined by $$\ell_A(w):=\ell(w)-|N(w^{-1})\cap A|.$$

Edgar then showed \cite[Theorem 2.3]{Edgar}

\begin{theorem}[Edgar]\label{thm:edgar}
Let $A\subseteq\T$. The following are equivalent
\begin{enumerate}
\item $\leq_A$ is a partial order,
\item $A\in\Bic$,
\item $\ell_A(xt)<\ell_A(x)$ for all $x\in \W$, $t\in (N(x^{-1})+A)$.
\end{enumerate}
\end{theorem}

Note that Edgar gives additional characterizations involving an analogue of the Bruhat graph -- we do not state them here since we will not require them. 

For $A\in \Bic$, using Edgar's Theorem we may define $\leq_A$ by $u\leq_A v$ if and only if there exists a sequence $u=u_0, u_1,\dots, u_k=v$, $u_i\in\W$ such that for any $i=0,\dots, k-1$, $\ell_A(u_{i+1})>\ell_A(u_i)$ and $u_{i+1}=u_i t_i$ for some $t_i\in\T$.

This generalizes the twisted Bruhat order because finite biclosed sets turn out to be exactly inversion sets of elements (see \cite[2.3]{Dyerweak}). Hence it follows from the above description that the $y$-twisted Bruhat order $\leq_y$ from Subsection \ref{sub:twisted} coincides with $\leq_{N(y)}$ for any $y\in \W$. In particular for finite Coxeter groups, we get nothing new. However for infinite groups one gets additional orders which may look very different from the twisted Bruhat orders as the following example shows

\begin{example}\label{ex:biclos}
Let $\W=\left\langle s, t\right\rangle$ be an infinite dihedral group with $\S=\{s,t\}$. Then the set $A:=\{s, sts, ststs, \dots\}$ lies in $\Bic$. One then has $$\dots~<_A tsts <_A sts <_A ts <_A s <_A e <_A t <_A st <_A tst <_A stst <_A \dots,$$
that is, $<_A$ is a total order in that case. 
\end{example}

The set $\Bic$ can be much bigger in general than $\{N(w)~|~w\in\W\}$ in the case where $\W$ is infinite. For more on the combinatorics of $\Bic$ we refer to \cite{Dyerweak}, \cite{HL}.

%Notice that the $y$-twisted Bruhat order $\leq_y$ coincides with $\leq_{N(y)}$; indeed, finite biclosed sets coincide with inversion sets (see references). 

We can summarize the characterizations of the cover relations in $\leq_A$ as follows:

\begin{lemma}\label{lemma:edgar}
Let $A\in\Bic$, $x\in\W$, $t\in\T$. The following are equivalent
\begin{enumerate}
\item $xt <_A x$,
\item $\ell_A(xt)<\ell_A(x)$,
\item $t\in (N(x^{-1})+A)$.
\end{enumerate}
\end{lemma}
\begin{remark}\label{rem:parity}
Given $x\in\W$, $t\in\T$, by definition of $\ell_A$ we have that $\ell_A(x)$ and $\ell_A(tx)$ have distinct parities. By Lemma \ref{lemma:edgar} it follows that either $tx <_A x$ or $tx >_A x$.
\end{remark}

\begin{remark}\label{rmq:distanceun}
For later use we shall mention that if $x\in\W$, $s\in \S$, then $sx <_A x$ if and only if $\lA(sx)=\lA(x)-1$. 
\end{remark}
Notice that the property in the above remark fails in general if me multiply $x$ by $s$ on the right (see Example~\ref{ex:biclos}). The following well-known property of the Bruhat order will turn out to be crucial when introducing twisted filtrations of Soergel bimodules in Section \ref{sec:twistedfiltrations}:

\begin{proposition}[Deodhar's "Z"-property]\label{prop:z}
Let $A\in\mathscr{B}(\Phi^+)$. Let $x,y\in \W$, $s\in\S$ such that $sx\leq_A x$, $sy\leq_A y$. Then $x\leq_A y$ iff $sx\leq_A y$ iff $sx\leq_A sy$.
\end{proposition}
Notice that the above Proposition is shown by Dyer for sets which are initial sections of reflection orders (see \cite[Proposition 1.9]{Dyershellings2}) which as mentioned at the beginning of the subsection are conjectured to be exactly the biclosed sets. The proof below is similar to Dyer's one, which generalizes to $\leq_A$. 
\begin{proof}
Assume that $x\leq_A y$. Then $sx\leq_A y$. Conversely, assume that $sx\leq_A y$. By definition of $\leq_A$ and Lemma \ref{lemma:edgar}, there exist $t_1,\dots, t_n\in\T$ such that 
$$sx=y t_1\cdots t_n <_A \dots <_A y t_1 <_A y.$$ Setting $s_i:= y t_i y^{-1}$ we get 
$$sx=s_1\cdots s_ny <_A \dots  <_A s_1 y <_A y.$$
Replacing $s_i$ by $s_1 s_2\cdots s_i\cdots s_1$ we can assume that 
$$sx=s_n\cdots s_1y <_A \dots  <_A s_1 y <_A y.$$
We claim that if $s\neq s_i$ for $i=1,\dots, n$, then $s s_i\cdots s_1y <_A s s_{i-1}\cdots s_1 y$. Indeed, assume that $s s_i\cdots s_1y >_A s s_{i-1}\cdots s_1 y$. Set $q:=s_{i-1}\cdots s_1 y$, $t:= q^{-1} s_i q$. Then the above condition becomes $sqt >_A sq$. By Lemma \ref{lemma:edgar}, it implies that $t\notin (N(q^{-1}s)+A)$. Firstly, assume that $t\notin N(q^{-1}s)$ and $t\notin A$. By Lemma \ref{prop:dyercaract} (2b), it follows that 
\begin{equation}\label{cond:t}
t\notin N(q^{-1})+ q^{-1} \{s\} q.
\end{equation}
But we have $s_i q=q t <_A q$ by the above sequence, hence $t\in (N(q^{-1})+A)$ by Lemma \ref{lemma:edgar}. Since $t\notin A$ it follows that $t\in N(q^{-1})$. Together with condition \ref{cond:t} it forces $t\in q^{-1}\{s\} q$, hence $t=q^{-1} s q$, hence $s_i=s$. If $t\in N(q^{-1}s)\cap A$ we similarly show that $s=s_i$. 

It follows that if $s\neq s_i$ for all $i$, then $$sx<_A x=ss_n\cdots s_1 y <_A \dots <_A ss_1 y <_A sy <_A y,$$
hence $x\leq_A y$. Hence assume that $s=s_i$ for some $i$, and take $i$ to be maximal with $s=s_i$. Then 
$$x= s s_n\cdots s_1 y<_A \dots <_A  s s_{i+1}\cdots s_1y <_A s s_{i}\cdots s_1y = s_{i-1}\cdots s_1 y <_A y,$$
hence $x\leq_A y$. Similarly we have $sx\leq_A y$ if and only if $sx\leq_A sy$.

\end{proof}
For $s\in\S$, we set $$R_{A,s}^{\uparrow}:=\{x\in\W~|~sx>_A x\},~R_{A,s}^{\downarrow}:=\{x\in\W~|~sx<_A x\}.$$
Note that $R_{A,s}^{\uparrow}=s R_{A,s}^{\downarrow}$, $\W=R_{A,s}^{\uparrow}\overset{\cdot}{\cup} R_{A,s}^{\downarrow}$.
\begin{corollary}\label{cor:ordretotal}
Let $A\in\Bic$, $s\in\S$. There exists an interval $I$ in $\mathbb{Z}$ and a total order $\{w_i\}_{i\in I}$ on $\W$ which refines $<_A$ and such that $R_{A,s}^{\uparrow}=\{w_i~|~i\in I,i~\text{even}\}$, and $sw _i= w_{i+1}$ for $i\in I$ even.  
\end{corollary}
\begin{proof}
Let $\dots, u_i, u_{i+1},\dots$ be a linear extension of $(R_{A,s}^{\uparrow}, <_A)$ and for any $i\in\mathbb{Z}$ such that $u_i$ is an element of $\W$, set $w_{2i}:=u_i$. Then set $w_{2i+1}:= s w_{2i}$. This gives en enumeration $\dots, w_i, w_{i+1},\dots$ of the group with  $R_{A,s}^{\uparrow}=\{w_i~|~i~\text{even}\}$, and $sw _i= w_{i+1}$ for $i$ even. It remains to check that it is compatible with $<_A$. 

Hence let $w_i <_A w_j$. We must show that $i<j$. If both $i$ and $j$ are even, then it follows from the fact that we enumerated the elements of $(R_{A,s}^{\uparrow}, <_A)$ in way compatible with $<_A$. If both $i$ and $j$ are odd, then we have $sw_i <_A w_i$, $sw_j <_A w_j$, and $w_i <_A w_j$. Proposition \ref{prop:z} implies that $w_{i-1}=sw_i <_A sw_j=w_{j-1}$, hence that $i-1 < j-1$ since both $w_{i-1}$ and $w_{j-1}$ lie in $R_{A,s}^{\uparrow}$. If $i$ is even and $j$ is odd, then $w_i <_A w_{i+1}= s w_i$, $s w_j=w_{j-1} <_A w_j$, hence by Proposition \ref{prop:z} we have $w_i <_A w_{j-1}$, hence since both lie in $R_{A,s}^{\uparrow}$ we have $i<j-1 <j$. If $i$ is odd and $j$ is even, then $w_{i-1}=s w_i <_A w_i <_A w_j <_A sw_j$, hence $i-1<j$ since $w_{i-1}, w_j\in R_{A,s}^{\uparrow}$, but since both $i-1$ and $j$ are even we get $i<j$.  
\end{proof}
By shifting the indices by one we of course get a total order with $R_{A,s}^{\uparrow}$ in odd positions.
\begin{definition}
Let $A\in\Bic$, $s\in\S$. A total order on $\W$ as in Corollary \ref{cor:ordretotal} is called \defn{$(A,s)$-compatible}.
\end{definition}

%%%%%%%%%%%%%%%%%%%%%%%%%%%%%%%%%%%%%%%%%%
%HECKE ALGEBRA AND GOOD PERMUTATION BRAIDS
%%%%%%%%%%%%%%%%%%%%%%%%%%%%%%%%%%%%%%%%%%

\section{Mikado braids and generalized standard bases}

\subsection{Standard and canonical bases of Hecke algebras} 

\begin{definition}
Let $(\W,\S)$ be a Coxeter system. For $s,t\in \S$, denote by $m_{st}$ the order
of $st$ in $\W$. Recall that $m_{st}=m_{ts}$. The \defn{Iwahori-Hecke algebra} $\H=\H(\W,\S)$ of $(\W,\S)$ is the associative, unital $\mathbb{Z}[v,v^{-1}]$-algebra with generators (as algebra) given by $\{T_s\mid s\in \S\}$ with relations
\begin{center}
$\begin{array}{ll}
T_s^2=(v^{-2}-1) T_s+ v^{-2} & \forall s\in \S,\\
\underbrace{T_s T_t\cdots}_{m_{st}~\text{factors}}=\underbrace{T_t T_s\cdots}_{m_{ts}~\text{factors}}, & \forall s,t\in \S.
\end{array}$
\end{center}
\end{definition}

Let $s_{1}\cdots s_k$ be a reduced expression of $w\in \W$. Then the element $T_{s_1}\cdots T_{s_k}$ is independent of the chosen reduced expression, and is therefore denoted by $T_w$. The algebra $\H$ turns out to be a free $\mathbb{Z}[v, v^{-1}]$-module and $\{T_w\}_{w\in \W}$ is a basis of it, usually called \defn{standard basis} of $\H$. Thanks to the first defining relation of $\H$, any of the $T_s$, and therefore any of the $T_w$ is invertible in $\H$. The set $\{T_{w^{-1}}\}_{w\in\W}$ is the \defn{costandard basis} of $\H$.

There is a unique semilinear involution $\bar~:\mathcal{H}\rightarrow\mathcal{H}$ such that $\overline{v}=v^{-1}$, $\overline{T_w}=(T_{w^{-1}})^{-1}$. Set $H_w:=v^{\ell(w)} T_w$. Then Kazhdan and Lusztig prove {\cite[Theorem 1.1]{KL}}

\begin{theorem}[Kazhdan-Lusztig]\label{thm:KL}
Let $(\W,\S)$ be any Coxeter system. 
\begin{enumerate}
\item For any $w\in \W$, there is a unique element $C'_w\in\mathcal{H}$ such that $\overline{C'_w}=C'_w$ and $C'_w\in H_w+\sum_{y< w} v\mathbb{Z}[v] H_y.$
\item  For any $w\in \W$, there is a unique element $C_w\in\mathcal{H}$ such that $\overline{C_w}=C_w$ and $C_w\in H_w+\sum_{y< w} v^{-1}\mathbb{Z}[v^{-1}] H_y.$
\end{enumerate}

\end{theorem}
It follows that $\{C_w\}_{w\in\W}$ and $\{C_w'\}_{w\in\W}$ are bases of $\H$, called \defn{canonical bases}. The coefficients of $C_w'$ of $C_w$ when expressed in the standard basis became known as \defn{Kazhdan-Lusztig polynomials}. Kazhdan and Lusztig conjectured that these polynomials have nonnegative coefficients, which was proven recently in full generality by Elias and Williamson \cite{EW} as a corollary of Soergel's conjecture. It corresponds to the positivity statement (\ref{p2}) from the Introduction with $y=e$. 

The algebra $\H$ has a unique semilinear involution $j_{\H}$ such that $j_{\H}(v)=v^{-1}$ and $j_\H(T_{s})=-v^2 T_{s}$ for all $s\in \S$. The bases $\{C_w\}$ and $\{C'_w\}$ are then related by the equalities 
\begin{equation}\label{eq:cc'}
C_w=(-1)^{\ell(w)} j_\H(C'_w),~\forall w\in \W.
\end{equation}
A reference for this material is \cite[Section 7.9]{Humph}.

In Subsection \ref{sub:generalstd} we introduce a family of bases of $\H$ indexed by elements of $\Bic$, such that the biclosed set $\emptyset$ corresponds to the standard basis while $\T$ corresponds to the costandard basis.

\subsection{Artin-Tits group and Mikado braids}

Let $\Br=B(\W,\S)$ denote the Artin-Tits group attached to the Coxeter system $(\W,\S)$. Recall that $\Br$ is generated by a set $\{\bs~|~s\in\S\}$ subject only to the braid relations. Denote by $\bw$ the positive canonical lift of $w\in\W$ in $\Br$, that is, the element $\bs_1 \bs_2\cdots \bs_k\in \Br$ where $s_1 s_2\cdots s_k$ is a reduced expression of $w$. By Tits-Matsumoto's lemma it is independent of the choice of the reduced expression. There is a canonical surjection $p:\Br\twoheadrightarrow \W$ defined by $p(\bs)=s$ for each $s\in\S$. Notice that $p(\bw)=w$. In case $\W$ is finite, the set $\{\bw~|~w\in\W\}$ is the set of prefixes of the Garside element $\Delta=\bw_0$ also known as set of \defn{simple} elements. 

Given $A\in\Bic$ and $x\in\W$, Dyer defines \cite[Proof of Lemma $9.1$]{Dyernil} an element $x_A\in\Br$ as follows. Let $s_1 s_2\cdots s_k$ be a reduced expression of $x$. Then set $$x_A:=\bs_1^{\varepsilon_1} \bs_2^{\varepsilon_2}\cdots\bs_k^{\varepsilon_k}\in\Br,$$ 
where for all $i=1,\dots, k$, 
$$
\varepsilon_i = \left\{
    \begin{array}{ll}
        -1 & \mbox{if } s_k s_{k-1}\cdots s_i\cdots s_k \in A \\
        1 & \mbox{otherwise.}
    \end{array}
\right.
$$
Notice that $\sum_{i=1}^k \varepsilon_i=\ell_A(x)$. Lifted reduced expressions as above might be considered as reduced expressions associated to a biclosed set $A$. In case $A=\emptyset$, all the exponents are positive and $x_\emptyset$ is the positive canonical lift $\bx$ of $x$ in $\Br$, in which case we recover the classical reduced expressions of elements of the Coxeter group.

Properties of these elements are given in the following lemma \cite[Sections 9.1 and 9.4]{Dyernil}
\begin{lemma}[Dyer]\label{lem:Mikado}
Let $A\in\Bic$, $x,y\in\W$. Then
\begin{enumerate}
\item The element $x_A$ is independent of the chosen reduced expression for $x$. 
\item One has $x_{\emptyset}=\bx$, $(x^{-1})_\T=\bx^{-1}$.
\item One has $(xy^{-1})_{N(y)}=\bx \by^{-1}$. 
\item One has $(x^{-1}y)_{\T\backslash{N(y^{-1})}}=\bx^{-1}\by$.
\end{enumerate}
\end{lemma}
\begin{proof}
Point (1) is \cite[Lemma 9.1]{Dyernil}. Point (2) follows directly from the definition, since in case $A=\emptyset$ (resp. $A=\T$) all the exponents in the definition of $x_A$ are positive (resp. negative). The property of point (3) is mentioned by Dyer in \cite[Section 9.4]{Dyernil} and (4) is the analogue for complements of inversion sets. We provide a proof of (3) here for completeness. We argue by induction on $\ell(xy^{-1})$. The claim is clearly true if $xy^{-1}=e$. Now let $s\in\S$ such that $sxy^{-1}<xy^{-1}$. Then $$(xy^{-1})_{N(y)}=\bs^{\varepsilon} (sxy^{-1})_{N(y)}= \bs^{\varepsilon}(sx)_{\emptyset} y_\emptyset^{-1},$$
where $\varepsilon=1$ if $yx^{-1}sxy^{-1}\notin N(y)$ and $\varepsilon=-1$ otherwise and the last equality above holds by induction. Assume that $\varepsilon=1$. Then $x^{-1}s x\notin N(y^{-1})$. But since $sxy^{-1}<xy^{-1}$, we have $s\in N(xy^{-1})=N(x)+ x N(y^{-1}) x^{-1}$. This forces $s\in N(x)$, hence $sx<x$, hence $\bs (sx)_{\emptyset} = x_{\emptyset}$. It follows that $(xy^{-1})_{N(y)}=x_{\emptyset} y_{\emptyset}^{-1}$ and the claim holds. Now assume that $\varepsilon=-1$. It follows that $x^{-1}s x\in N(y^{-1})$ and together with $s\in N(xy^{-1})=N(x)+ x N(y^{-1}) x^{-1}$ we get that $s\notin N(x)$, hence $sx>x$. We therefore have $(sx)_{\emptyset}=\bs x_{\emptyset}$ and it follows that $(xy^{-1})_{N(y)}=x_{\emptyset} y_{\emptyset}^{-1}$ again.

\end{proof}

\begin{lemma}\label{lem:multformartin}
Let $w\in\W$, $s\in\S$ , $A\in\Bic$ such that $sw >_A w$. Then $$\bs (w_A)=(sw)_A.$$
\end{lemma}
\begin{proof}
By assumption we have that 
\begin{equation}\label{eq:2}
w^{-1}sw\notin (N(w^{-1})+A).
\end{equation}
Firstly, assume that $w<sw$. It follows that $w^{-1}sw\notin N(w^{-1})$, hence by (\ref{eq:2}) that $w^{-1}sw\notin A$. Choosing a reduced expression $s_1\cdots s_k$ for $w$, we have that $s s_1\cdots s_k$ is a reduced expression for $sw$. It follows by definition of $(sw)_A$ that $(sw)_A= \bs^{\varepsilon_s} w_A$ where $\varepsilon_s=-1$ if $w^{-1}s w\in A$ and $\varepsilon_s=1$ otherwise. But $w^{-1}sw\notin A$, hence $\varepsilon_s=1$. 

Now assume that $sw<w$.  It follows that $w^{-1}sw\in N(w^{-1})$, hence by (\ref{eq:2}) that $w^{-1}sw\in A$. Choosing a reduced expression  $s_1\cdots s_k$ for $sw$, we have that $s s_1\cdots s_k$ is a reduced expression for $w$. It follows by definition of $w_A$ that $w_A= \bs^{\varepsilon_s} (sw)_A$ where $\varepsilon_s=-1$ if $w^{-1}s w\in A$ and $\varepsilon_s=1$ otherwise. But $w^{-1}sw\in A$, hence $\varepsilon_s=-1$. 
\end{proof}

If $A$ is finite, then there exists $y\in\W$ such that $A=N(y)$. Hence in that case we have $x_{N(y)}=(xyy^{-1})_{N(y)}=(xy)_{\emptyset} (y_{\emptyset})^{-1}$. In particular for finite $\W$, any $x_A$ for $x\in\W$ and $A\in\Bic$ can be written in the form $\bz \by^{-1}$ for $y,z\in \W$. Moreover, if $\W$ is finite we have $\T\backslash N(y^{-1})=N(y^{-1}w_0)$, hence in that case any $x_A$ can also be written in the form $\bu^{-1} \bv$ (see also \cite[Proposition 4.3]{DG} where a Garside theoretic interpretation is also given). Conversely elements of the form $\bz \by^{-1}$ or $\bu^{-1} \bv$ may be written in the form $x_A$ by the above lemma (without assumption on $\W$). Hence if $\W$ is finite and $\beta\in\Br$ the three following conditions are equivalent
\begin{equation}\label{p8}
\exists A\in \Bic, x\in\W\text{ s.t. }\beta=x_A,
\end{equation}
\begin{equation}\label{p9}
\exists z, y\in\W\text{ s.t. }\beta=\bz \by^{-1},
\end{equation}
\begin{equation}\label{p10}
\exists u, v\in\W\text{ s.t. }\beta=\bu^{-1}\bv.
\end{equation}

For arbitrary Coxeter groups, condition (\ref{p8}) is weaker than both (\ref{p9}) and (\ref{p10}) in general. Moreover (\ref{p9}) and (\ref{p10}) are not equivalent in general (see Example \ref{exple:mik} below).

\begin{definition}
We say that $\beta\in\Br$ is a \defn{Mikado braid} if it satisfies Condition (\ref{p8}). We say that $\beta\in\Br$ is a \defn{left Mikado braid} (resp. \defn{right Mikado braid}) if it satisfies Condition (\ref{p9}) (resp. (\ref{p10})). This generalizes the definition given in \cite{DG} in case $\W$ is finite. The terminology comes from a geometric characterization of these braids in type $A_n$ in terms of Artin braids (see \cite[Definition 5.4]{DG}). 
\end{definition}

\begin{example}\label{exple:mik}
Let $\W=\left\langle r, s, t \right\rangle$ be a universal Coxeter system with $\S=\{r,s,t\}$. Then $\Br$ is the free group on the three generators $\br$, $\bs$, $\bt$. Hence the right Mikado braid $\br^{-1} \bs$ cannot be written on the form $\bz \by^{-1}$ with $z, y\in\bf{\W}$. The following example was pointed out by M.~Dyer. Consider the hyperplane $H$ in $V$ spanned by the positive root $\alpha_r$ of $r$ and the positive root $2\alpha_r+2\alpha_s+\alpha_t$ of $rstsr$. Consider the open half space defined by $H$ and containing $\alpha_s$, namely
$$H^{+}:=H+ \mathbb{R}_{>0} \alpha_s,$$
and define $A:=H^{+}\cap\Phi^+$. Then $A\in\Bic$. Let $x=tsr$. Then $$x_A=\bt \bs^{-1} \br.$$
Indeed, both positive roots of $r$ and $rstsr$ lie in $H$ while the positive root $2\alpha_r+\alpha_s$ of $rsr$ lies in $H^{+}$. But since $\Br$ is a free group $x_A$ is neither a left nor a right Mikado braid. 
\end{example}

\subsection{Generalized standard bases of Hecke algebras}\label{sub:generalstd}

The aim of this subsection is to generalize the standard basis by adding a dependency in a biclosed set $A\in\Bic$, namely to define a basis $\{T_{w,A}\}_{w\in\W}$ of $\H$ for any $A\in\Bic$. 

Since the generators $T_s$, $s\in\S$ of $\H$ satisfy the braid relations, there is a surjective homomorphism $$h:\mathbb{Z}[v, v^{-1}] [\Br]\rightarrow \H$$ such that $h(\bs)=T_s$ for any $s\in\S$. Let $A\in\Bic$, $x\in\W$ and define $T_{x,A}:=h(x_A)$. In other words, we associate to any reduced expression $x=s_1\cdots s_k$ of $x\in W$ the element $$T_{x,A}:=T_{s_1}^{\varepsilon_1} T_{s_2}^{\varepsilon_2}\cdots T_{s_k}^{\varepsilon_k}\in\H,$$
where for all $i=1,\dots, k$, 
$$
\varepsilon_i = \left\{
    \begin{array}{ll}
        -1 & \mbox{if } s_k s_{k-1}\cdots s_i\cdots s_k \in A \\
        1 & \mbox{otherwise.}
    \end{array}
\right.
$$

\begin{lemma}\label{lem:genbase}
Let $A\in\Bic$. 
\begin{enumerate}
\item The set $\{T_{w,A}\}_{w\in\W}$ is a basis of $\H$.
\item For $x\in\W$ he have $T_{x,\emptyset}=T_x$, $T_{x, \T}=T_{x^{-1}}^{-1}$. Hence for $A=\emptyset$, $\{T_{w,A}\}_{w\in\W}$ is the standard basis while for $A=\T$ it is the costandard basis. 
\item For $x,y\in\W$ we have $T_x T_y^{-1}=T_{xy^{-1}, N(y)}$.
\item For $x,y\in\W$ we have $T_x^{-1} T_y=T_{x^{-1}y, \T\backslash N(y^{-1})}$.
\end{enumerate}
\end{lemma}
\begin{proof}
Using that $T_s^{-1}=v^2 T_s +v^2-1$ for any $s\in \S$ and expanding $T_{w,A}$ in the standard basis, we get an upper triangular matrix with invertible coefficients on the diagonal passing from $\{T_w\}_{w\in\W}$ to $\{T_{w,A}\}_{w\in\W}$ if we order the common indexing set $\W$ by any linear extension of the Bruhat order on $\W$. This proves $(1)$. The two other statements follow immediately from Lemma \ref{lem:Mikado} $(2)$-$(4)$.
\end{proof}
The following follows immediately from Lemma \ref{lem:multformartin}
\begin{lemma}\label{lem:multform}
Let $w\in\W$, $s\in\S$ , $A\in\Bic$ such that $sw >_A w$. Then $$T_s T_{w,A}=T_{sw, A}.$$
\end{lemma}

%%%%%%%%%%%%%%%%%%%%%%%%%%%%%%%%%%%%%%%%%%
%TWISTED FILTRATIONS ON SOERGEL BIMODULES
%%%%%%%%%%%%%%%%%%%%%%%%%%%%%%%%%%%%%%%%%%

\section{Twisted filtrations of Soergel bimodules}\label{sec:twistedfiltrations}

\subsection{Soergel bimodules}

Let $(\W, \S)$ be a Coxeter system and $V$ a reflection faithful representation over $\mathbb{R}$ in the sense of \cite[Definition 1.5]{S}. Let $R:=\mathcal{O}(V)\cong S(V^*)$ be the coordinate ring of $V$. In particular $R$ comes equipped with a $\mathbb{Z}$-graduation with the convention that $\mathrm{deg}(V^*)=2$.
%Let $\H=\H(\W,\S)$ be the Hecke algebra, that is, the associative, unital $\mathbb{Z}[v,v^{-1}]$-algebra with generators $T_s$, $s\in S$ satisfying the braid relations as well as the quadratic relations $T_s^2=(v^{-2}-1) T_s+v^{-2}$ for any $s\in \S$. It has a well-known standard basis $\{T_w\}_{w\in\W}$ and two canonical bases $\{C_w\}_{w\in\W}$ and $\{C'_w\}_{w\in\W}$.

Let $\mathcal{R}$ denote the category of $\mathbb{Z}$-graded $R\otimes_{\mathbb{R}} R$-modules which are finitely generated from the left and from the right. The category $\mathcal{R}$ is a monoidal category via $\otimes_R$. It satisfies the Krull-Schmidt property (see \cite[Remark 1.3]{S}). For $M\in\mathcal{R}$ and $i\in\mathbb{Z}$ denote by $M_i$ the homogeneous component of degree $i$ of $M$. Given $M\in\mathcal{R}$ and $i\in\mathbb{Z}$, we denote by $M(i)$ the element of $\mathcal{R}$ equal to $M$ as $R\otimes_{\mathbb{R}} R$-module with shifted graduation, that is, such that $M(i)_j=M_{i+j}$ for any $j\in\mathbb{Z}$. Here $M_i$, $i\in\mathbb{Z}$ denotes the subspace of homogeneous elements of degree $i$ of $M$. We define the \defn{graded rank} of a finitely generated graded right $R$-module $M$ as the element of $\mathbb{Z}[v, v^{-1}]$ given by $\mathrm{\underline{rk}} M:=\mathrm{\underline{dim}} (M/ M R_{>0})$, where for a finite-dimensional $\mathbb{Z}$-graded vector space $U$ we denote by $\mathrm{\underline{dim}}(U)$ the graded dimension $\sum_{i\in\mathbb{Z}} (\dim U_i) v^i\in\mathbb{Z}[v, v^{-1}]$ of $U$. We denote by $\mathrm{\overline{\underline{rk}}} M$ the graded rank of $M$ after substitution of $v$ by $v^{-1}$. 

Given $B, B'\in\mathcal{R}$, we denote by $\mathrm{Hom}(B,B')$ the morphisms in the category $\mathcal{R}$, that is, the homomorphisms of bimodules $B\rightarrow B'$ which are homogeneous of degree zero. We furthermore set
$$\mathrm{Hom}^{\bullet}(B,B'):=\bigoplus_{i\in\mathbb{Z}} \mathrm{Hom}(B, B'(i)).$$
Notice that it comes equipped with a structure of graded right $R$-module. 

To any $s\in \S$ we associate the bimodule $B_s:=R\otimes_{R^s} R(1)\in\mathcal{R}$, where $R^s\subseteq R$ is the graded subring of $s$-invariant functions. For $x\in W$, we denote by $R_x$ the element of $\mathcal{R}$ equal to $R$ as left $R$-module but with right action twisted by $x$. Denote by $\left\langle \mathcal{R}, \otimes_R \right\rangle$ the split Grothendieck ring of $\mathcal{R}$, endowed with a $\mathbb{Z}[v,v^{-1}]$-algebra structure via $v\cdot \left\langle M \right\rangle= \left\langle M(1) \right\rangle$ for $M\in\mathcal{R}$. Soergel shows \cite[Theorems 1.10 and 5.3]{S}
\begin{theorem}[Soergel's categorification theorem]\label{thm:soergelcat}
Let $(\W,\S)$ be a Coxeter system. 
\begin{enumerate}
\item There is a unique homomorphism of $\mathbb{Z}[v, v^{-1}]$-algebras $$\mathcal{E}:\mathcal{H}\rightarrow \left\langle \mathcal{R}, \otimes_R \right\rangle$$ such that $\mathcal{E}(v)=\left\langle R[1] \right\rangle$ and $\mathcal{E}(vT_s+v)=\left\langle B_s \right\rangle$ for any $s\in\S$.
\item The homomorphism $\mathcal{E}$ has a left inverse $$\mathrm{ch}:\left\langle \mathcal{R},\otimes_R\right\rangle\rightarrow \mathcal{H}$$ given by $\mathrm{ch}(\left\langle B\right\rangle)=\sum_{x\in\mathcal{W}} \mathrm{\overline{\underline{rk}}} (\mathrm{Hom}(B, R_x)) T_x,$ for $B\in\mathcal{R}$.
\end{enumerate}
\end{theorem}
It implies that $\H$ is isomorphic to the split Grothendieck ring $\left\langle\mathcal{B}\right\rangle$ of the additive monoidal category $\mathcal{B}$ generated by tensor products $B_s \otimes B_t\cdots \otimes B_u$ (called \defn{Bott-Samelson} bimodules; here $st\cdots u$ is any finite word in the elements of $\S$) and stable by direct sums, direct summands and graduation shifts (so that $\left\langle\mathcal{B}\right\rangle$ is a $\mathbb{Z}[v, v^{-1}]$-algebra). By definition an object of $\B$ is a \defn{Soergel bimodule}. Hence indecomposable Soergel bimodules are (shifts of) indecomposable direct summands of tensor products $B_s\otimes B_t\otimes\cdots\otimes B_u$. For simplicity we may write tensor products over $R$ by juxtaposition.    

Soergel shows that indecomposable Soergel bimodules are (up to shifts and isomorphism) indexed by elements of $\W$ (see \cite[Theorem 6.16]{S}). The indecomposable bimodule $B_w$ indexed by $w\in\W$ may be described as follows: let $st\cdots u$ be a reduced expression for $w$. Then there is a unique indecomposable direct summand $B_w$ of $B_s B_t\cdots B_u$ which does not occur as a direct summand of such a tensor product for elements $y$ with $\ell(y)<\ell(w)$. The Bott-Samelson bimodule $B_s B_t\cdots B_u$ depends on the reduced expression chosen for $w$, but it turns out that the direct summand $B_w$ does not. 

\begin{remark} As the notation suggests, the indecomposable bimodule indexed by $s\in\S$ should be the bimodule $B_s=R\otimes_{R^s} R(1)$ already defined before. Indecomposability may be seen in this specific case as follows: the ring $R$ is nonnegatively graded with $R_0=\mathbb{R}$. Hence $R\otimes_{R^s} R$ is nonnegatively graded with a one-dimensional degree zero component, namely $(R\otimes_{R^s} R)_0 =\mathbb{R} (1\otimes 1)$. Hence in any decomposition $R\otimes_{R^s} R=M\oplus N$ as graded bimodule, $1\otimes 1$ must lie either in $M$ or in $N$. But since $1\otimes 1$ generates $R\otimes_{R^s} R$ as bimodule either $M=0$ or $N=0$, hence $B_s$ is indecomposable. 
\end{remark}
Soergel's conjecture was proven by Elias and Williamson \cite{EW}:

\begin{theorem}[Soergel's conjecture]\label{sconj}
For any $w\in\W$, we have $\mathcal{E}(C_w')=\left\langle B_w\right\rangle$.
\end{theorem}

\subsection{Twisted support filtrations}
Let $\mathcal{B}$ denote the category of Soergel bimodules attached to the Coxeter system $(\W,\S)$. Recall that $V$ denotes a reflection faithful representation of $(\W,\S)$ over $\mathbb{R}$.

Since $R\otimes_\mathbb{R} R\cong \mathcal{O}(V)\otimes_\mathbb{R}\mathcal{O}(V)\cong\mathcal{O}(V\times V)$, a (finitely generated) $R\otimes_\mathbb{R} R$-module is the same as a (quasi-)coherent sheaf on $V\times V$. 

For any subset $U\subseteq\W$ and any $B\in\mathcal{R}$ we consider as in \cite{S} the subbimodule $\Gamma_U(B)$ of $B$ containing those elements whose support (i.e., the support of the coherent sheaf corresponding to the subbimodule generated by the element) is included in $\mathrm{Gr}(U)=\bigcup_{w\in u} \mathrm{Gr}(w)$, where $$\mathrm{Gr}(w):=\{(wv,v)~|~v\in V\}\subseteq V\times V.$$ For $i\in\mathbb{Z}$, $A\in\Bic$, we set
$$\Gamma_{\geq i}^A(B):=\Gamma_{\{w\in \W|\ell_A(w)\geq i\}}(B).$$
We are interested in the bimodules $B\in\mathcal{R}$ supported in some $\mathrm{Gr}(U)$ for $U\subseteq\W$, $U$ finite, such that for each $i$ the subquotient $$\Gamma_{\geq i / > i}^A(B):=\gi(B)/\gii(B)$$ is isomorphic to a finite direct sum of shifted copies of $R_x$ for various $x\in \W$ such that $\lA(x)=i$. We denote by $\mathcal{F}_{\Delta}^A$ the full subcategory of $\mathcal{R}$ having these bimodules as objects. In the sequel we also introduce the notation $\Delta_x^A:=R_x(-\ell_A(x))$. The aim is to show that any object of $\mathcal{B}$ lies in $\FA$. For $B\in\FA$, we denote by $[B:\Delta_x^A(i)]_A$ the multiplicity of $\Delta_x^A(i)$ in the filtration. 
\begin{proposition}\label{prop:eq}
Let $B\in\mathcal{B}$, $s\in \S$. If $B\in\FA$, then $B_s\otimes B$ also lies in $\FA$, and for $x\in\W$ such that $sx>_A x$ and $i\in\mathbb{Z}$ we have the following equalities
$$[B_s\otimes B:\Delta_{sx}^A(i)]_A=[B:\Delta_{sx}^A(i+1)]_A+[B:\Delta_x^A(i)]_A.$$
$$[B_s\otimes B:\Delta_{x}^A(i)]_A=[B:\Delta_{sx}^A(i)]_A+[B:\Delta_x^A(i-1)]_A.$$
\end{proposition}
\begin{proof}
Since by Proposition \ref{prop:z} the order $<_A$ satisfies Deodhar's property Z of Coxeter groups, one can argue exactly as in the proof of \cite[Proposition 5.7]{S} using also Remarks \ref{rmq:distanceun} and \ref{rem:parity}. 
\end{proof}

\begin{remark}
Let us say a bit more on the importance of these filtrations in Soergel's approach. The above filtration generalizes Soergel's \defn{standard filtration} \cite[Proposition 5.7]{S}, which corresponds to the case where $A=\emptyset$. Soergel showed that $\mathcal{B}\subseteq\mathcal{F}_{\Delta}^\emptyset$ and that given $B\in\mathcal{B}$, the map $\mathrm{ch}$ from Theorem \ref{thm:soergelcat} may be described as
\begin{equation}\label{stand}
\mathrm{ch}(\left\langle B\right\rangle)=\sum_{x\in\W}{\sum_{i\in\mathbb{Z}} [B:\Delta_x^\emptyset(i)]_\emptyset v^{i+\ell(x)}} T_{x}.
\end{equation}
Soergel also introduced a \defn{costandard filtration} corresponding to the case where $A=\T$, and furthermore showed that the map $\mathrm{ch}$ from Theorem \ref{thm:soergelcat} may be described as
\begin{equation}\label{costand}
\mathrm{ch}(\left\langle B\right\rangle)=\sum_{x\in\W}{\sum_{i\in\mathbb{Z}} [B:\Delta_x^\T(i)]_\T v^{i-\ell(x)}} T_{x^{-1}}^{-1}.\end{equation}
Hence if Soergel's conjecture holds, i.e. setting $B=B_w$ and $C'_w=\mathrm{ch}(\left\langle B_w\right\rangle)$, the Kazhdan-Lusztig polynomials can be interpreted as graded multiplicities in these filtrations (hence in particular they are nonnegative). 

\end{remark}

\begin{corollary}
We have $\mathcal{B}\subseteq \FA$.
\end{corollary}
\begin{proof}
The functor $B\mapsto \gi(B)/\gii(B)$ preserves direct sums since both $\gi(-)$ and $\gii(-)$ do. Together with the fact that the Krull-Schmidt property is satisfied in $\mathcal{R}$, it follows that any direct summand of $B\in\FA$ also lies in $\FA$. Using Proposition \ref{prop:eq} inductively we have that any Bott-Samelson bimodule lies in $\FA$. But indecomposable objects in $\mathcal{B}$ are up to isomorphism all (shifts of) direct summands of Bott-Samelson bimodules. 
\end{proof}

\begin{definition}
For $A\in\Bic$, $B\in\mathcal{B}$, we call the filtration of $B$ by support twisted by $A$ as introduced above the \defn{A-filtration} of $B$. 
\end{definition}

\begin{remark}[Refining the $A$-filtration]\label{rem:refine}
Arguing as in \cite[Lemma 6.3]{S}, we can refine our filtration as follows. Let $\{w_i\}_{i\in I}$ be a linear extension of $<_A$, where $I$ is an interval in $\mathbb{Z}$. Writing $\Gamma_{\geq w_j}^A$ for $\Gamma_{\{w_i\in\W| i\geq j\}}^A$ we then consider for $B\in\mathcal{B}$ the subbimodule $B^j:=\Gamma_{\geq w_j}^A B$. We then get a filtration
$$\cdots \subseteq B^{j+1} \subseteq B^j \subseteq B^{j-1}\subseteq \cdots$$
such that $\Gamma_{\geq w_j/ > w_j}^A(B):= B^j / B^{j+1}\cong \bigoplus_{k} \Delta_{w_j}^A(n_k)$ and the graded multiplicities turn out to be the same as in the $A$-filtration of $B$. In particular this is independent of the chosen linear extension of $<_A$. The key point is that there is no nontrivial extension between $R_x$ and $R_y$ for incomparable $x$ and $y$ in $<_A$; indeed Soergel showed that one can have a nontrivial extension only if $x$ and $y$ differ by a reflection (use \cite[Lemma 5.8]{S} together with the isomorphism $R_x\cong \mathcal{O}(\mathrm{Gr}(x))$), which by Remark \ref{rem:parity} cannot happen for incomparable elements in $<_A$.  

\end{remark}

For later use we need to understand how the support of a bimodule in $\FA$ behaves when tensoring with a $B_s$, $s\in\S$. The proof of the following is similar to \cite[Proof of 5.7]{S} and \cite[Proposition 6.5]{Will} for the Bruhat order.

\begin{proposition}\label{prop:mpair}
Let $A\in\Bic$, $s\in\S$. Let $\{w_i\}_{i\in I}$ be an $(A,s)$-compatible total order on $\W$ for some interval $I\subseteq\mathbb{Z}$. For $B\in\FA$ and $m$ even we have the following isomorphisms of graded bimodules
\begin{enumerate}
\item $\Gamma_{\geq w_m}^A (B_s B)\cong B_s (\Gamma_{\geq w_m}^A B)$,
\item $\Gamma_{\geq w_m/ >w_{m+1}}^A (B_s B)\cong B_s (\Gamma_{\geq w_m / > w_{m+1}}^A B)$.
\end{enumerate} 

\end{proposition}
\begin{proof}
Consider the short exact sequence $$0\longrightarrow \Gamma_{\geq w_m}^A B \longrightarrow B \longrightarrow B/\Gamma_{\geq w_m}^A B \longrightarrow 0.$$
Since $B\in\FA$, we have that $\Gamma_{\geq w_m}^A B$ (resp. $B/\Gamma_{\geq w_m}^A B$) has an $A$-filtration with subquotients given by shifts of $\Delta_{w_j}^A$ for $j \geq m$ (resp. $j<m$). Note that both $\{w_m, w_{m+1},\dots\}$ and $\{\dots w_{m-2}, w_{m-1}\}$ are $s$-stable since our total order is $(A,s)$-compatible and $m$ is even. Now since $R$ is a free right $R^s$-module of rank two we have that the functor $B_s\otimes -:\mathcal{R}\rightarrow \mathcal{R}$ is exact hence tensoring the short exact sequence above we get another short exact sequence 
\begin{equation}\label{eq:ses}
0\longrightarrow B_s(\Gamma_{\geq w_m}^A B) \longrightarrow B_sB \longrightarrow B_s(B/\Gamma_{\geq w_m}^A B) \longrightarrow 0.
\end{equation}

Now $R\otimes_{R^s} R_{w_j}$ is indecomposable and is a nontrivial extension of $R_{sw_j}$ and $R_{w_j}$ as can be seen by tensoring the short exact sequence $$0\longrightarrow R_s(-2)\longrightarrow R\otimes_{R^s} R\longrightarrow R\longrightarrow 0$$
on the right by $R_{w_j}$ (which defines an invertible functor). Hence $B_s(\Gamma_{\geq w_m}^A B)$ has a filtration with subquotients given by (shifts) of $R_{w_j}$ for $j\geq m$, while $B_s(B/\Gamma_{\geq w_m}^A B)$ has a filtration with subquotients given by (shifts) of $R_{w_j}$ for $j <m$. But since $B_s B\in\FA$ and any nonzero element in $R_{w_j}$ generates a subbimodule isomorphic to $R_{w_j}$ as ungraded bimodule, hence has support $\mathrm{Gr}(w_j)$, it follows that $B_s(\Gamma_{\geq w_m}^A B)$ is exactly the subbimodule of $B_s B$ containing the elements whose support is included in $\bigcup_{j\geq m} \mathrm{Gr}(w_j)$. The short exact sequence (\ref{eq:ses}) can therefore be rewritten as
$$0\longrightarrow \Gamma_{\geq w_m}^A (B_s B) \longrightarrow B_sB \longrightarrow B_sB/ \Gamma_{\geq w_m}^A(B_s B) \longrightarrow 0,$$
and $\Gamma_{\geq w_m}^A (B_s B)\cong B_s (\Gamma_{\geq w_m}^A B)$ which proves $(1)$. Now considering the exact sequence 
$$ 0\longrightarrow B_s(\Gamma_{\geq w_{m+2}}^A B)\longrightarrow B_s(\Gamma_{\geq w_{m}}^A B)\longrightarrow B_s(\Gamma_{\geq w_{m}/ > w_{m+1}} B)\longrightarrow 0,$$
we get $(2)$ using the isomorphisms $$\Gamma_{\geq m}^A (B_s B)\cong B_s (\Gamma_{\geq w_m}^A B),~\Gamma_{\geq w_{m+2}}^A (B_s B)\cong B_s (\Gamma_{\geq w_{m+2}}^A B).$$
\end{proof}

\subsection{Three-parameter Kazhdan-Lusztig positivity}
The aim of this subsection is to prove: 
\begin{theorem}[Three-parameter Kazhdan-Lusztig positivity]\label{threeparam}
Let $w\in \W$, $A\in\Bic$. Write $C'_w=\sum_{x\in\W} h_{x,w}^A T_{x,A}$. Then $${h_{x,w}^A}=\sum_{i\in\mathbb{Z}} [B_w:\Delta_x^A(i)]_A v^{i+\ell_A(x)}.$$ 
In particular the generalized Kazhdan-Lusztig polynomials $h_{x,w}^A\in\mathbb{Z}[v,v^{-1}]$ have nonnegative coefficients.
\end{theorem}
The proof of the theorem will occupy the remainder of the section. Theorem \ref{threeparam} was proven by Dyer and Lehrer for finite Weyl groups \cite[Theorem 2.8]{DL}. Dyer then showed \cite[Conjecture 7(b)]{Dyerrep} (relying on partially unpublished result) that the theorem would hold for finite Coxeter groups if Soergel's conjecture holds. Dyer and Lehrer and Dyer proved the theorem in the following form: 

\begin{corollary}
Let $w,y\in\W$. Then $$C'_w T_y\in\sum_{z\in\W} \mathbb{Z}_{\geq 0}[v^{\pm 1}] T_z.$$ 
\end{corollary}
\begin{proof}
For any $x\in\W$, we have that $T_{xy} T_y^{-1}=T_{x,N(y)}$ by Lemma \ref{lem:genbase} $(3)$, hence this comes here as a particular case of Theorem \ref{threeparam}.
\end{proof}

\begin{remark}\label{rmq:multungraded}
Let $\varphi:\H\rightarrow \mathbb{Z}[\W]$ denote the specialization morphism $v\mapsto 1$. For any $x\in\W$, $A\in\Bic$ we have $\varphi(T_{x,A})=x$. For $A,B\in\Bic$, $x,w\in\W$ specializing $C_w'$ in Theorem \ref{threeparam} we get the equality $$\sum_{i\in\mathbb{Z}} [B_w:\Delta_x^A(i)]_A=\sum_{i\in\mathbb{Z}} [B_w:\Delta_x^B(i)]_B.$$
In particular, forgetting the grading we have that the ungraded multiplicites of the bimodule twisted by $x$ coincide in all twisted filtrations. This should remain the reader about twisted Verma modules: in the finite Weyl group case, our modules $\Delta_x$ might be thought of as corresponding to the twisted Verma modules (see \cite{AndLaur}) which have the same character as their corresponding untwisted Verma module but different module structure, while $B_w$ is the analogue of a projective module. The suitable category in which to have an analogue of a twisted standard module is rather the bounded homotopy category $\Kb$ of Soergel's category in which one has a Rouquier complex associated to $T_{x,A}$ (see Section \ref{sec:libwil}). 
\end{remark}
As a consequence of the remark above (set $B=\emptyset$) or by using the fact that a Bott-Samelson bimodule $B_s B_t\cdots B_u$ has support included in $\mathrm{Gr}(\{e,s\}\{e,t\}\cdots \{e,u\})$ we have:

\begin{lemma}\label{lem:mutiplbruhat}
Let $w\in\W$, $A\in\Bic$. We have $$[B_w: \Delta_x^A(i)]\neq 0 \Rightarrow x\leq w.$$
\end{lemma}

We first prove:

\begin{proposition}\label{prop:bottsamelson}
Let $st\cdots u$ be a reduced expression of $w\in\W$, $A\in\Bic$. Write $C_s' C_t'\cdots C_u'=\sum_{x\in\W} q_{x,st\cdots u}^A T_{x,A}$. Then 
$${q_{x,st\cdots u}^A}=\sum_{i\in\mathbb{Z}} [B_s B_t\cdots B_u:\Delta_x^A(i)]_A v^{i+\ell_A(x)}.$$
\end{proposition}
\begin{proof}
The proof is by induction on $\ell(w)$. If $w=e$, then the linear expansion of $1=T_e=T_{e,A}$ in the basis $\{T_{x,A}\}_{x\in\W}$ is simply given by $1=T_{e,A}$. For $x\neq e$ we have $[R:R_x(i)]_A=0$ for any $i\in\mathbb{Z}$. In case $x=e$ we have $[R:R(i)]_A\neq 0$ if and only if $i=0$ in which case the multiplicity is equal to $1$. Since $\ell_A(e)=0$ the claimed formula holds in that case. 

Now let $\ell(w)>0$ and assume that the claimed formula holds for $w'=sw$. Let $x\in \W$ such that $sx>_A x$. It implies that $\ell_A(sx)=\ell_A(x)+1$ (see Remark \ref{rmq:distanceun}). A direct computation using the equality $C'_s=vT_s+v$ together with Lemma \ref{lem:multform} shows that the coefficient $q_{sx,st\cdots u}^A$ of $T_{sx, A}$ in $C_s' C_t'\cdots C_u'$ is equal to $v q_{x, t\cdots u}^A+v^{-1} q_{sx,t\cdots u}^A$. For short we write $B'$ for $B_t\cdots B_u$ and $B$ for $B_s B'$. Now by induction we have that ${v q_{x, t\cdots u}^A+v^{-1} q_{sx,t\cdots u}^A}$ is equal to
\begin{eqnarray*}
& &v \sum_{i\in\mathbb{Z}}[B':\Delta_x^A(i)]_A v^{i+\ell_A(x)}+v^{-1} \sum_{i\in\mathbb{Z}}[B':\Delta_{sx}^A(i)]_A v^{i+\ell_A(sx)}\\
&=&\sum_{i\in\mathbb{Z}} [B':\Delta_x^A(i)]_A v^{i+\ell_A(x)+1}+[B':\Delta_{sx}^A(i+1)]_A v^{i+1+\ell_A(sx)-1}\\
&=&\sum_{i\in\mathbb{Z}} ([B':\Delta_x^A(i)]_A+[B':\Delta_{sx}^A(i+1)]_A) v^{i+\ell_A(sx)},\\
\end{eqnarray*}
And applying the first equality in Proposition \ref{prop:eq} we get that $${q_{sx,st\cdots u}^A}=\sum_{i\in\mathbb{Z}} [B:\Delta_{sx}^A(i)]_A v^{i+\ell_A(sx)},$$ as claimed. Similarly, the coefficient $q_{x,w}^A$ of $T_{x,A}$ in $C_s' C_t'\cdots C_u'$ is equal to $v q_{x,t\cdots u}^A+v^{-1} q_{sx,t\cdots u}^A$. As above we have by induction that ${v q_{x, t\cdots u}^A+v^{-1} q_{sx,t\cdots u}^A}$ equals
\begin{eqnarray*}
& &v\sum_{i\in\mathbb{Z}}[B':\Delta_x^A(i)]_A v^{i+\ell_A(x)}+v^{-1} \sum_{i\in\mathbb{Z}}[B':\Delta_{sx}^A(i)]_A v^{i+\ell_A(sx)}\\
&=&\sum_{i\in\mathbb{Z}} [B':\Delta_x^A(i-1)]_A v^{i-1+\ell_A(x)+1}+[B':\Delta_{sx}^A(i)]_A v^{i+\ell_A(sx)-1}\\
&=&\sum_{i\in\mathbb{Z}} ([B':\Delta_x^A(i-1)]_A+[B':\Delta_{sx}^A(i)]_A) v^{i+\ell_A(x)},\\
\end{eqnarray*}
And applying the second equality in Proposition \ref{prop:eq} we get that $${q_{x,st\cdots u}^A}=\sum_{i\in\mathbb{Z}} [B:\Delta_{x}^A(i)]_A v^{i+\ell_A(x)},$$ as claimed.
\end{proof}

\begin{proof}[Proof of Theorem \ref{threeparam}]

Let $st\cdots u$ be a reduced expression for $w$. Let us write a Krull-Schmidt decomposition of $B_s B_t\cdots B_u$ as $B_w\oplus\bigoplus_{z<w} P_z\cdot B_z$. Here the $P_z\in\mathbb{Z}[v,v^{-1}]$ are polynomials $\sum_{i\in\mathbb{Z}} n_i v^i$ with nonnegative coefficients and we use the notation $\left(\sum_{i\in\mathbb{Z}} n_i v^i\right)\cdot B:=\bigoplus_{i\in\mathbb{Z}} B(i)^{\oplus n_i}$ for a graded bimodule $B$. It follows from Soergel's categorification theorem (Theorem \ref{thm:soergelcat}) and Soergel's conjecture (Theorem \ref{sconj}) that $$C_s' C_t'\cdots C_u'=C_w'+\sum_{z<w} P_z C_z'.$$

The proof is by induction on $\ell(w)$. The case where $\ell(w)=0$ or even $\ell(w)=1$ is already contained in Proposition \ref{prop:bottsamelson}. Hence assume $\ell(w)>1$. Because the functor $B\mapsto \gi(B)/\gii(B)$ is additive we have that $[-,\Delta_x^A(i)]$ is additive, hence 
\begin{eqnarray*}
C_s'C_t'\cdots C_u' &=&\sum_{x\in\W}{\sum_{i\in\mathbb{Z}} [B_sB_t\cdots B_u:\Delta_x^A(i)]_A v^{i+\ell_A(x)}} T_{x,A}\\
&=&\sum_{x\in\W}{\sum_{i\in\mathbb{Z}} \left([B_w:\Delta_x^A(i)]_A+\sum_{z<w} {P_z}[B_z:\Delta_x^A(i)]_A\right) v^{i+\ell_A(x)}} T_{x,A}\\
&=&\sum_{x\in\W}{\sum_{i\in\mathbb{Z}} [B_w:\Delta_x^A(i)]_A v^{i+\ell_A(x)}} T_{x,A}+\sum_{z<w} P_z C'_z,\\
\end{eqnarray*}
Where the last equality holds by induction. But we also have that 
\begin{eqnarray*}
C_s' C_t'\cdots C_u'=C_w'+\sum_{z<w} P_z C_z',
\end{eqnarray*}
Hence comparing the two equalities we get that 
\begin{eqnarray}\label{eq:threeparam}
C_w'=\sum_{x\in\W}{\sum_{i\in\mathbb{Z}} [B_w:\Delta_x^A(i)]_A v^{i+\ell_A(x)}} T_{x,A}.
\end{eqnarray}
\end{proof}

\begin{corollary}\label{cor:mult}
Let $w\in\W$, $A\in\Bic$. Then 
$$
[B_w:\Delta_w^A(i)]_A = \left\{
    \begin{array}{ll}
        1 & \mbox{if } i=0\\
        0 & \mbox{otherwise.}
    \end{array}
\right.
$$
\end{corollary}
\begin{proof}
By Theorem \ref{thm:KL}, the coefficient of $T_w$ when $C_w'$ is expressed in the standard basis is $v^{\ell(w)}$. By Lemma \ref{lem:mutiplbruhat}, the multiplicity $[B_w:\Delta_x^A(i)]_A$ can be nonzero only if $x\leq w$. Therefore in the expansion of $C_w'$ in the basis $\{T_{x,A}\}_{x\in\W}$, the only element which can contribute a $T_w$ is $T_{w,A}$ since $T_{x,A}$ is a linear combination of $T_y$ for $y\leq x$. Expanding $T_{w,A}$ in the standard basis, we get that the coefficient of $T_w$ is given by $v^{2 N}$, where $N$ is the number of negative exponents in the definition of $w_A$. But by definition of $w_A$, this number $N$ is precisely equal to $|N(w^{-1})\cap A|$. Multipliying by the coefficient $h_{w,w}^A\in\mathbb{Z}_{\geq 0}[v,v^{-1}]$ of $T_{w,A}$ in the expansion of $C_w'$ in the basis $\{T_{x,A}\}_{x\in\W}$, we therefore must have $$h_{w,w}^A v^{2|N(w^{-1})\cap A|}= v^{\ell(w)},$$
hence $h_{w,w}^A=v^{\ell_A(w)}$. Hence by Theorem \ref{threeparam} we have $$v^{\ell_A(w)}={h_{w,w}^A}=\sum_{i\in\mathbb{Z}} [B_w:\Delta_w^A(i)]_A v^{i+\ell_A(w)},$$
which implies the result. 
\end{proof}

One can also combinatorially derive the more general positivity statement
\begin{corollary}\label{cor:doubletwist}
Let $w,y\in \W$, $A\in\Bic$. Then $$C'_wT_{y,A}\in\sum_{x\in\W} \mathbb{Z}_{\geq 0}[v,v^{-1}] T_{x,A}.$$ 
\end{corollary}
\begin{proof}
Arguing as in the proof of statement~$(3)$ of Lemma~\ref{lem:Mikado}, one shows for $y,z\in\W$, $A\in\Bic$ the more general formula $$(zy)_A=z_{B} y_A,$$ where $B=
N(y)+yAy^{-1}$. Notice that $B\in\Bic$ by \cite[Proposition 3.6]{Edgar}. Taking the image in the Hecke algebra, the claimed property is hence equivalent to the statement
$$C'_w\in\sum_{z\in\W} \mathbb{Z}_{\geq 0}[v,v^{-1}] T_{z,B}$$
which follows immediately from Theorem \ref{threeparam}.

\end{proof}

\begin{remark}
For $y\neq e$, the product $C'_wT_{y,A}$ is not the class of an object in $\mathcal{B}$. It nevertheless appears as natural to try to interpret the coefficients in Corollary~\ref{cor:doubletwist} as multiplicities in a canonical filtration of a bimodule. Dyer developped \cite{Dyerrep} an approach where he considers a category with more objects including the $\Delta_x^A$. The Hecke algebra is then categorified not as an algebra but as a left module over itself by considering the tensor product of indecomposable Soergel bimodules with (twisted) standard modules. Such an approach is similar to the categorification of the Hecke algebra of a finite Weyl group by indecomposable projective functors acting on (a graded version of) the principal block of the BGG category $\mathcal{O}$ (see \cite{Stroppel}, \cite[Section 7.5]{Maz} and the references therein). In Dyer's framework, the above coefficients should be the muliplicities of $B_w\otimes_R \Delta_y$ in an $A$-twisted filtration. Dyer's approach appears as more natural for further generalizations involving biclosed sets of roots. For our purposes, twisted filtrations of Soergel bimodules will suffice. Alternatively, if one wants to have a canonical object for $C'_wT_{y,A}$ in a category keeping a monoidal structure, one can consider the bounded homotopy category of $\mathcal{B}$ (see Sections \ref{sec:libwil} and \ref{sec:lin} below).      
%The statement of Corollary \ref{cor:doubletwist} might remind the reader about indecomposable projective functors on the principal block of the (graded version of) the BGG category $\mathcal{O}$ acting on (twisted) Verma modules. Standard modules (and their twisted versions) do not make sense in Soergel's category (but rather in its bounded homotopy category). Dyer developped \cite{Dyerrep} a framework which generalizes this approach and 
\end{remark}

\section{Generalized Libedinsky-Williamson formula}\label{sec:libwil}

The aim of this section is to compute the $A$-filtration of a Rouquier complex for a Mikado braid $x_A$ where $x\in\W$, $A\in\Bic$. \subsection{Rouquier complexes and minimal complexes}\label{sec:rouquier}
We denote by $K^b(\mathcal{R})$ (resp. $\Kb$) the homotopy category of bounded complexes of bimodules in $\mathcal{R}$ (resp. $\mathcal{B}$). The monoidal structure on $\mathcal{R}$ and $\mathcal{B}$ induces a monoidal structure on the corresponding homotopy categories which we will simply denote by juxtaposition. Given a complex
$$A~:~~\dots\rightarrow {^{i-1}A}\rightarrow {^{i}A}\rightarrow {^{i+1}A}\rightarrow \dots~\in\Kr,$$
We denote by $A[j]$, where $j\in\mathbb{Z}$, the complex $A$ shifted in homological degree by $j$, that is, such that $^i A[j]={^{j+i}A}$. Let us consider the following indecomposable complexes in $\Kb$:

\begin{equationarray*}{ccccccccccc}
F_s &=& & & 0 &\rightarrow & B_s & \overset{\mu_s}{\rightarrow} & R(1) & \rightarrow & 0, \\
E_s &=& 0 & \rightarrow & R(-1) & \overset{\eta_s}{\rightarrow} & B_s & \rightarrow & 0 & &,
 \end{equationarray*}
where $B_s$ sits in both cases in homological degree zero. Here $\eta_s(r)=\frac{1}{2}(r\otimes f_s+rf_s\otimes 1)$ for any $r\in R$ and $\mu_s$ is the multiplication map. The functors $F_s\otimes -$ and $E_s\otimes -$ define mutually inverse equivalences of $\Kr$, categorifying the braid group of the Coxeter system in finite type (see \cite{Rouq}, \cite{Rouqprep}, \cite{Rouq1}, \cite{Jensen} and the references therein). Given $x\in \W$ with reduced expression $st\cdots u$, we write $F_x$ for the Rouquier complex corresponding to the positive canonical lift $\bx$ of $x$ in the braid group, that is, $F_x=F_s F_t\cdots F_u$. We set $E_x:=(F_{x^{-1}})^{-1}=E_s E_t\cdots E_u$. Notice that these complexes do not depend on the choice of reduced expression up to canonical isomorphism in $\Kb$ (see \cite[Section 9.3.1]{Rouq}). We replace $F_x$ by a \defn{minimal} complex in $\Kb$, that is, a complex obtained from $F_x$ by removing contractible direct summands (see \cite[Section 6.1]{EW}; a minimal complex is isomorphic to the starting complex in $\Kb$ and any two miminal complexes turn out to be isomorphic as complexes of bimodules). In general we will always assume $F_x$ to be a miminal Rouquier complex.

\subsection{Libedinsky-Williamson formula for Mikado braids}

The aim of this section of to prove

\begin{proposition}\label{prop:libwill}
Let $(\W,\S)$ be a Coxeter system. Let $x\in\W$, $A\in\Bic$. Let $\beta:=x_A\in\Br$ and let $F_{\beta}\in\Kb$ be a minimal Rouquier complex for $\beta$. Let $I\subseteq\mathbb{Z}$ be an interval such that $\{w_i\}_{i\in I}$ is an enumeration of $\W$ refining $<_A$. There are isomorphisms in $\Kb$

$$
\Gamma_{\geq w_i / > w_i}^A (F_{\beta})\cong \left\{
    \begin{array}{ll}
        \Delta_{x}^A & \mbox{if } w_i=x, \\
        0 & \mbox{otherwise.}
    \end{array}
\right.
$$

\end{proposition}

\begin{proof}
The proof is by induction on $\ell(x)$. If $\ell(x)=1$, then $x\in\S$ and $F_\beta$ is either isomorphic to $F_s$ to to $E_s$ depending on whether $e <_A s$ or $s <_A e$. Since the indecomposable Soergel bimodules occurring in these complexes are supported only on $\mathrm{Gr}(e)$ and $\mathrm{Gr}(s)$, we can as well replace $A$ by $\emptyset$ if $e <_A s$ and $A$ by $\T$ if $s <_A e$, in which case we can compute it by hands or refer to Libedinsky-Williamson's formula \cite[Proposition 3.10]{LW}.

Hence assume that the claimed formula holds for $x$ with $\ell(x) > 0$. Let $s\in\S$ such that $sx < x$. By induction the claim holds for $\beta'=(sx)_A$. By Remark \ref{rem:refine} we can replace our total order by any other total order refining $<_A$, hence assume that $\{w_i\}_{i\in I}$ is an $(A,s)$-compatible order (see Corollary \ref{cor:ordretotal}). Assume that $sx <_A x$, the other case being similar. By Lemma \ref{lem:multformartin}, we have that $\bs (sx)_A = x_A$. Hence $$F_\beta\cong F_s F_{\beta'}~\text{in } \Kb.$$
By Proposition \ref{prop:mpair} for any $i\in 2\mathbb{Z}$ we have an isomorphism in $\Kb$
\begin{equation}\label{eq:homotopy}
F_s(\Gamma_{\geq w_i/ >w_{i+1}}^A (F_{\beta'}))\cong \Gamma_{\geq w_i/ > w_{i+1}}^A (F_\beta).
\end{equation}
We have $sx=w_k$ for some $k\in 2\mathbb{Z}$ and $x=w_{k+1}$. In that specific case by induction and (\ref{eq:homotopy}) we have 
\begin{eqnarray*}
\Gamma_{\geq w_k/ > w_{k+1}}^A (F_\beta)&\cong& F_s (\dots\rightarrow 0\rightarrow \Delta_{sx}^A \rightarrow 0\rightarrow \dots)\\
&\cong& \dots\rightarrow 0\rightarrow B_s \Delta_{sx}^A \rightarrow \Delta_{sx}^A(1)\rightarrow 0\rightarrow\dots 
\end{eqnarray*}
where in both complexes the smallest nonzero degree is homological degree zero. Remembering the short exact sequence $$0\rightarrow R_x(-1)\rightarrow B_s R_{sx} \rightarrow R_{sx}(1)\rightarrow 0$$

we get that $$\Gamma_{\geq w_{k+1}}^A (\dots\rightarrow 0\rightarrow B_s \Delta_{sx}^A \rightarrow \Delta_{sx}^A(1)\rightarrow 0\rightarrow\dots) \cong R_x(-\ell_A(sx)-1)\cong \Delta_x^A,$$
where the last equality holds since $sx >_A x$, hence $\ell_A(x)=\ell_A(sx)+1$ (see Remark \ref{rmq:distanceun}). Setting $i=k+1$ this implies the claim for $w_i=x$ since there is an isomorphism of functors 
$$\Gamma_{\geq w_{k+1}}^A( \Gamma_{\geq w_k / > w_{k+1}}^A (-))\cong \Gamma_{\geq w_{k+1} / >w_{k+1}}^A(-).$$
(The above isomorphism of functors holds for the (reverse) Bruhat order as a Corollary of \cite[Lemma 3.1]{LW}, but this Lemma adapts mutatis mutandis if we replace Bruhat order by $<_A$). For $w_i=sx$, applying $\Gamma_{\geq w_{k}/ > w_k}^A(-)$ to the complex $(\dots\rightarrow 0\rightarrow B_s \Delta_{sx}^A \rightarrow \Delta_{sx}^A(1)\rightarrow 0\rightarrow\dots)$
yields $$\dots\rightarrow 0 \rightarrow \Delta_{sx}^A(1)\overset{\mathrm{id}}{\rightarrow}\Delta_{sx}^A(1)\rightarrow0\rightarrow\dots$$
which is homotopic to zero, whence the result using the isomorphism of functors
$$\Gamma_{\geq w_{k}/ > w_k}^A( \Gamma_{\geq w_k / > w_{k+1}}^A (-))\cong \Gamma_{\geq w_{k} / >w_{k}}^A(-).$$
Now let $z\in \W$, $x\neq z$ with $sz <_A z$. Let $j\in 2\mathbb{Z}$ such that $sz=w_j$, $z=w_{j+1}$. One has an exact sequence of complexes $$0\rightarrow \Gamma_{\geq w_{j+1}/ >w_{j+1}}^A (F_{\beta'})\rightarrow \Gamma_{\geq w_j/ >w_{j+1}}^A (F_{\beta'})\rightarrow \Gamma_{\geq w_j/ >w_{j}}^A (F_{\beta'})\rightarrow 0.$$
Any indecomposable summand in the left (resp. in the right) term is isomorphic to a shift of $R_{z}$ (resp. $R_{sz}$). But by \cite[Lemma 5.8]{S}, any nontrivial extension between two such modules is isomorphic to a shift of $\mathcal{O}(\mathrm{Gr}(z)\cup \mathrm{Gr}(sz))$ which splits as left $R^s$-module by \cite[Lemma 4.5 (1)]{S}. Hence tensoring the above short exact sequence by $B_s$ and viewing the resulting complex in $\Kb$ we get an exact triangle
$$B_s \Gamma_{\geq w_{j+1}/ >w_{j+1}}^A (F_{\beta'})\rightarrow B_s \Gamma_{\geq w_j/ >w_{j+1}}^A (F_{\beta'})\rightarrow B_s\Gamma_{\geq w_j/ >w_{j}}^A (F_{\beta'})\overset{[1]}{\rightarrow},$$
and by induction the first and the third term of the complex are homotopic to zero. Hence the middle term is also homotopic to zero. But this middle term is isomorphic to $\Gamma_{\geq w_j/ >w_{j+1}}^A (B_s F_{\beta'})$ by Proposition \ref{prop:mpair}. Arguing as for the case where $w_i=x$ or $sx$ we conclude that 
\begin{equation}\label{eq:jj1}
\Gamma_{\geq w_{j+1}/ >w_{j+1}}^A (B_s F_{\beta'})\cong 0 \cong \Gamma_{\geq w_j/ >w_{j}}^A (B_s F_{\beta'}).
\end{equation}
The result now follow by applying $\Gamma_{\geq w_{j+1}/ >w_{j+1}}^A(-)$ and $\Gamma_{\geq w_{j}/ >w_{j}}^A(-)$ to the distinguished triangle 
$$F_s F_{\beta'}\rightarrow B_s F_{\beta'}\rightarrow F_{\beta'}(1)\overset{[1]}{\rightarrow}.$$
 Indeed, we showed that applying them yields a complex homotopic to zero in the middle position (see Equation (\ref{eq:jj1})), and the same holds for the complex on the right by induction. 
\end{proof}

\section{Linearity of Rouquier complexes}\label{sec:lin}

The aim of this section is to show that minimal Rouquier complexes of left and right Mikado braids are "linear", that is, that all the indecomposable summands in homological degree $i$ have the form $B_x(i)$ for all $i\in\mathbb{Z}$. Together with a statement on the parity of the elements $x$ such that $B_x$ occurs in homological degree $i$, this will allow us to generalize Property (\ref{p1}) and its analogue for right Mikado braids to arbitrary Coxeter groups. 

Following Elias and Williamson \cite{EW}, denote by $\Kbp$ the full subcategory of $\Kb$ whose objects are those complexes with minimal complex $F$ satisfying the following property: for any $i\in\mathbb{Z}$ such that $^i F$ is nonzero and any indecomposable summand $B$ in $^i F$, there exists $x\in\W$, $k\leq i$ such that $B\cong B_x(k)$. Similarly, denote by $\Kbn$ the full subcategory of $\Kb$ whose objects are those complexes with minimal complex satisfying the following property: for any $i\in\mathbb{Z}$ such that $^i F$ is nonzero and any indecomposable summand $B$ in $^i F$, there exists $x\in\W$, $k\geq i$ such that $B\cong B_x(k)$. In other words, $\Kb$ (resp. $\Kbn$) consists of those complexes with minimal complex $F$ having indecomposable Soergel bimodules with shifts at most (resp. at least) $i$ occurring in homological degree $i$. 

We recall some results of Elias and Williamson. 
\begin{lemma}
Let $x\in \W$, $s\in\S$. Let $B_x$ be an indecomposable Soergel bimodule viewed in $K^b(\mathcal{B})$, in homological degree zero. 

\begin{enumerate}
\item Assume that $xs <x$. Then $B_x F_s\cong B_x(-1)$ and $B_x E_s\cong B_x(1)$ in $\Kb$.
\item Assume that $xs>x$. Then $B_x F_s, B_x E_s\in\Kbp\cap \Kbn$.
\end{enumerate}
Similar statements hold if we multiply the $F_s$ and $E_s$ on the left instead. 
\end{lemma} 
\begin{proof}
The statements involving the complex $F_s$ are \cite[Lemma 6.5]{EW}. The proofs of the statements involving the $E_s$ are similar. 
\end{proof}
\begin{lemma}\label{lem:fses}
Let $F\in\Kb$, $s\in\S$.
\begin{enumerate}
\item If $F\in \Kbp$, then $FF_s\in\Kbp$. 
\item If $F\in\Kbn$, then $FE_s\in\Kbn$.
\end{enumerate}
The same conclusions hold if we multiply the $F_s$ and $E_s$ on the left instead.
\end{lemma}

\begin{proof}
The statements involving the complex $F_s$ are \cite[Lemma 6.6]{EW}. The proofs of the statements involving the $E_s$ are similar. 
\end{proof}
The following result is \cite[Corollary 6.7]{EW}
\begin{corollary}\label{cor:positivelift}
Let $w\in \W$. Then $F_w\in \Kbp$ and $E_w\in \Kbn$. 
\end{corollary}
Notice that more generally, a Rouquier complex for any positive braid lies in $\Kbp$. Elias and Williamson show in addition that $F_w\in \Kbn$. More precisely, they show the following

\begin{theorem}[{\cite[Theorem 6.9]{EW}}]\label{thm:ewlin}
Let $w\in\W$ with minimal Rouquier complex $F_w$. Then 
\begin{enumerate}
\item $^0 F_w=B_w$.
\item for $i\geq 1$, $^i F_w=\bigoplus B_z(i)^{\oplus m_{z,i}}$ for $z<w$ and $m_{z,i}\in\mathbb{Z}_{\geq 0}$. 
\end{enumerate}
In particular, $F_w\in\Kbp\cap\Kbn$. 
\end{theorem}
We can get an analogous result for the complexes $E_w$ by dualizing Elias and Williamson's argument, which will allow us to derive the linearity of Rouquier complexes of the braids mentioned at the beginning of the section:

\begin{proposition}\label{thm:costand}
Let $w\in\W$ with minimal Rouquier complex $E_w$. Then 
\begin{enumerate}
\item $^0 E_w=B_w$.
\item for $i\leq -1$, $^i E_w=\bigoplus B_z(i)^{\oplus m_{z,i}}$ for $z<w$ and $m_{z,i}\in\mathbb{Z}_{\geq 0}$. 
\end{enumerate}
In particular, $E_w\in\Kbp\cap\Kbn$.
\end{proposition}
We give the proof for the sake of completeness. It relies on the following Lemma (which is an analogue of \cite[Lemma 6.11]{EW} for the complexes $E_w$): 

\begin{lemma}
If $^i E_w$ contains a summand isomorphic to $B_z(j)$ for some $z<w$, then $^{i+1} E_w$ contains a summand isomorphic to $B_{z'}(j')$ with $j'>j$, $z'>z$.  
\end{lemma}
\begin{proof}
Notice that the $\nabla$-character as defined in \cite{EW} corresponds to our $A$-filtration for $A=\T$. We have that $[B_w:\Delta_w^\T]_\T=1$ while $[B_w:\Delta_w^\T(i)]_\T=0$ for $i\neq 0$ (see Corollary \ref{cor:mult}). It follows from Soergel's conjecture \ref{sconj} and (\ref{costand}) that for $z<w$, $[B_w:\Delta_z^\T(i)]_\T\neq 0$ implies that $i<0$ (we will say that the $\Delta^\T$-character is "negative"). Indeed, one has that 
$$C'_w\in H_w+\sum_{z<w} v\mathbb{Z}_{\geq 0}[v] H_z,$$
hence dualizing we get 
$$C'_w\in H_{w^{-1}}^{-1}+\sum_{z<w} v^{-1}\mathbb{Z}_{\geq 0}[v^{-1}] H_{z^{-1}}^{-1}.$$
Comparing with Theorem \ref{threeparam} we get the statement (notice that $H_{x^{-1}}^{-1}= v^{-\ell(x)} T_{x^{-1}}^{-1}=v^{-\ell(x)} T_{x,\T}=v^{\ell_\T(x)} T_{x,\T}$ for any $x\in\W$ by Lemma \ref{lem:genbase} (2)). 

Consider a summand $B_z(j)$ of $^i F_w$. Arguing exactly as in the first paragraph of the proof of \cite[Lemma 6.11]{EW}, we have that our summand $B_z(j)$ maps in the complex $E_w$ to a sum of $B_y(k)$ for $k>j$. Moreover, any nontrivial map to our summand $B_z(j)$ must come from a summand $B_y(k)$ with $k<j$. 

We now fix a total order on $\W$ refining $<_A$ and apply the corresponding subquotient functor $\Gamma_{\geq z/ >z}^\T$ to the complex $E_w$ for $z\in\W$. The resulting complex has a summand in $^i \Gamma_{\geq z/ >z}^\T (E_w)=\Gamma_{\geq z/ >z}^\T (^i E_w)$ isomorphic to $\Delta_z^\T(j)$. It follows from Proposition \ref{prop:libwill} that this summand must map isomorphically to some $\Delta_z^\T(j)$ in $\Gamma_{\geq z/ >z}^\T (^{i+1} E_w)$ or be mapped to isomorphically from a summand $\Delta_z^\T(j)$ in $\Gamma_{\geq z/ >z}^\T (^{i-1} E_w)$. Assume that the second holds. The summand isomorphic to $\Delta_z^\T(j)$ in homological degree $i-1$ of $\Gamma_{\geq z/ >z}^\T (E_w)$ mapping to our summand must be a subquotient of some summand $B_y(j')$ of $^{i-1} E_w$. By negativity of the $\Delta^\T$-character we must have $j'>j$. Hence our isomorphism
$$\Delta_z^\T(j)\overset{\sim}{\longrightarrow} \Delta_z^\T(j)$$
in $\Gamma_{\geq z/ >z}^\T (E_w)$ must be induced by a nontrivial map 
$$B_y(j')\longrightarrow B_z(j)$$
from the $i-1$\ts{th} homological degree to the $i$\ts{th} with $j'>j$, but we mentioned above that this is impossible. Hence our summand $\Delta_z^\T(j)$ in $\Gamma_{\geq z/ >z}^\T (^i E_w)$ must map isomorphically to a summand $\Delta_z^\T(j)$ in $\Gamma_{\geq z/ >z}^\T (^{i+1} E_w)$. This last summand comes as a subquotient of some summand $B_{z'}(j')$ from $^{i+1} E_w$. It follows that $z'>z$ (because any subquotient of a twisted filtration of $B_{z'}$ is indexed by some element which is lower than $z'$ in Bruhat order by Lemma \ref{cor:mult}) and our isomorphism is induced by a nontrivial map 
$$B_z(j)\longrightarrow B_{z'}(j')$$
from the $i$\ts{th} homological degree to the $i+1$\ts{th} with $j'>j$ and $z' > z$, which concludes. 
\end{proof}

\begin{proof}[Proof of Theorem \ref{thm:costand}]
It follows from the above Lemma that any summand of $^0 E_w$ is isomorphic to $B_w$ since $î E_w=0$ for $i>0$. Moreover, in fact we have $^0 E_w=B_w$ because by applying $\Gamma_{\geq w/ >w}^\T (F_w)$ we must get $\Delta_w^\T$ by \cite[Proposition 3.10]{LW}; hence such a summand must come from a $B_w$ since the indecomposable summands appearing in the homological degrees of $E_w$ are indexed by elements lower than or equal to $w$ in Bruhat order, and since no other summand $B_w$ can appear in negative degrees it follows that $^0 E_w=B_w$. Using the above lemma inductively we have that the shifts of the summands in homological degree $k\leq 0$  are at most equal to $k$, hence that $E_w\in\Kbp$. But we already know that $E_w\in\Kbn$ by Corollary \ref{cor:positivelift}.
\end{proof}

\begin{theorem}\label{thm:linearity}
Let $x,y\in\W$ and $\beta=\bx^{-1} \by\in\Br$. Let $F_\beta$ denote a minimal Rouquier complex for $\beta$. Then 

\begin{enumerate}
\item $^0 F_{\beta}$ has a unique summand isomorphic to $B_w$ where $w=x^{-1}y$. All other summands of $^0 F_\beta$ are isomorphic to $B_z$ for $z<w$. 
\item For any $i\in\mathbb{Z}\backslash\{0\}$, we have $^i F_\beta=\bigoplus B_z(i)^{\oplus m_{z,i}}$ for $z<w$ and $m_{z,i}\in\mathbb{Z}_{\geq 0}$. 
\end{enumerate}
In particular, we have $F_\beta\in\Kbn\cap\Kbp$. One has a similar statement for $\beta=\bx \by^{-1}$.
\end{theorem}
\begin{proof}
We only prove the statement for $\beta=\bx^{-1}\by$, the one for $\bx \by^{-1}$ is similar. A Rouquier complex for $\bx^{-1}\by$ is given by $E_{x^{-1}} F_y$. By Theorem \ref{thm:ewlin} we have that $F_y\in \Kbp\cap\Kbn$, in particular $F_y\in\Kbn$. Applying Lemma \ref{lem:fses} inductively we get that $E_{x^{-1}} F_y\in\Kbn$. By Theorem \ref{thm:costand}, we have that $E_{x^{-1}}\in\Kbp\cap\Kbn$, in particular $E_{x^{-1}}\in\Kbp$. Applying Lemma \ref{lem:fses} inductively we get that $E_{x^{-1}} F_y\in\Kbp$. Hence $E_{x^{-1}} F_y\in \Kbn\cap\Kbp$. 

%The same remains true when taking a direct summand of $E_{x^{-1}} F_y$ which is a minimal complex, and such a complex is isomorphic to $F_\beta$. 

Now writing $\beta$ as a lifted reduced expression of $w=x^{-1}y$, say $s_1^{\varepsilon_1}\cdots s_k^{\varepsilon_k}$ where $s_1\cdots s_k$ is a reduced expression of $w$, we get that the corresponding complex $K_{s_1}\otimes\cdots\otimes K_{s_k}$ where $K_{s_i}=F_{s_i}$ if $\varepsilon_i=1$, respectively $K_{s_i}=E_{s_i}$ if $\varepsilon_i=-1$ has exactly one summand isomorphic to $B_w$ which sits in homological degree zero, all other summands being indexed by elements lower than $w$ in Bruhat order (alternatively one can apply Theorem \ref{prop:libwill}). This summand can therefore not be suppressed when projecting to a minimal complex, whence the result. 
\end{proof}

\subsection{Proof of inverse positivity for left and right Mikado braids}\label{sub:inverse}

Let $(\W, \S)$ be a Coxeter system. The aim of this section is to prove
\begin{theorem}\label{thm:inversepos}
Let $x,y\in \W$. Then 
$$T_{x}^{-1} T_y\in\sum_{w\in\W} \mathbb{Z}_{\geq 0}[v^{\pm 1}] C_w.$$ 
Similarly, $T_x T_y^{-1}\in\sum_{w\in\W} \mathbb{Z}_{\geq 0}[v^{\pm 1}] C_w.$
\end{theorem}

\begin{proof}
We will only prove the first statement, the proof of the second one is similar. Using Relation (\ref{eq:cc'}) we can reformulate it as $$(-1)^{\ell(x)+\ell(y)}T_x^{-1} T_y\in\sum_{w\in\W} \mathbb{Z}_{\geq 0}[v^{\pm 1}] (-1)^{\ell(w)}C_w'.$$

Denoting by $\left\langle \Kb\right\rangle$ the Grothendieck group of the triangulated category $\Kb$, one has an isomorphism of abelian groups (see \cite{Rose}) $$\mathrm{cl}:\langle \Kb\rangle\overset{\sim}{\longrightarrow} \left\langle \mathcal{B}\right\rangle, C\mapsto \sum_{i\in\mathbb{Z}} (-1)^i\left\langle ^i C\right\rangle,$$
where $\left\langle {^i C} \right\rangle$ denotes the class of the bimodule ${^i C}\in\mathcal{B}$ in $\left\langle \mathcal{B}\right\rangle$. Moreover, the map is compatible with graduation shifts (where a complex $C(i)\in\Kb$ is defined as the complex $C$ where all summands of all homological degrees are shifted by $i$) and with the monoidal structure on $\Kb$ defined by the total tensor product of complexes (induced by the tensor product $\otimes_R$ on $\mathcal{B}$). That is, denoting by $\otimes$ the total tensor product in $\Kb$, it follows from the definition of total tensor product that for $C, D\in\Kb$ we have 
$$\mathrm{cl}(C\otimes D)=\left(\sum_{i\in\mathbb{Z}} (-1)^i \langle ^i C\rangle\right)\left(\sum_{j\in\mathbb{Z}} (-1)^j \langle ^j D\rangle\right)=\mathrm{cl}(C)\mathrm{cl}(D).$$

Hence $\mathrm{cl}$ is an isomorphism of $\mathbb{Z}[v, v^{-1}]$-algebras between $\langle \Kb\rangle$ and $\H$. In particular we have that $$T_x^{-1} T_y = v^{\ell(x)-\ell(y)}H_x^{-1} H_y=v^{\ell(x)-\ell(y)}\mathrm{cl}(E_{x^{-1}} F_y).$$
Indeed, we have $H_x= H_{s_1}\cdots H_{s_k}$ for any reduced expression $s_1 s_2\cdots s_k$ of $x\in \W$ and $$\mathrm{cl}(F_{s_i})=\langle B_{s_i}\rangle - \langle R(1)\rangle = C'_{s_i}-v= v T_{s_i}+v-v= H_{s_i}.$$

Since $(-1)^{\ell(x^{-1}y)}=(-1)^{\ell(x)+\ell(y)}$, Theorem \ref{thm:inversepos} follows if in the notations of Theorem \ref{thm:linearity} we prove that $m_{z,i}=0$ if $i$ and $\ell(z)-\ell(p(\beta))$ have different parity. Indeed, by definition of the map $\mathrm{cl}$ we have that the coefficient $g_z$ of $C_z'$ in $\mathrm{cl}(E_{x^{-1}} F_y)$ is given by $\sum_{i\in\mathbb{Z}} (-1)^{i}m_{z,i} v^i$. Our statement on the parity then implies that $$g_z=(-1)^{\ell(z)-\ell(x)+\ell(y)}\sum_{i\in\mathbb{Z}} m_{z,i} v^i,$$
which implies the result. 

The statement is proven by induction on $\ell(p(\beta))$. If it has length equal to $1$, then $F_\beta$ is equal to $F_s$ or $E_s$ for some $s\in\S$ and the claim is true in that case. Hence assume that the claim holds for right Mikado braids $\beta$ such that $\ell(p(\beta))<k$ and let $\beta=\bx^{-1}\by$ be such that $\ell(p(\beta))=k$. Let $\beta=s_1^{\varepsilon_1} s_2^{\varepsilon_2}\cdots s_k^{\varepsilon_k}$, where $s_1\cdots s_k$ is a reduced expression of $x^{-1}y$ and $\varepsilon_i=-1$ if $s_k \cdots s_i\cdots s_k\in A$ and $\varepsilon_i=1$ otherwise. 

Set $\beta':=s_2^{\varepsilon_2}\cdots s_k^{\varepsilon_k}$. Notice that $\ell(p(\beta'))=k-1$ and $\beta'=(s_1 x^{-1}y)_{\T\backslash N(y^{-1})}$, in particular it is also a right Mikado braid. Assume that $\varepsilon_1=1$, the other case being similar. Then we have $$^i (F_sF_{\beta'})\cong(B_s\otimes {^i F_{\beta'}})\oplus (R(1)\otimes {^{i-1} F_{\beta'}}).$$

%\cong (B_s\otimes (\bigoplus B_z(i)^{\oplus m_{z,i}}))\oplus (R(1)\otimes {\bigoplus B_z(i-1)^{\oplus m_{z,i-1}}}),$$

Notice that $F_s F_{\beta'}\cong F_\beta$ in $\Kb$, but $F_s F_{\beta'}$ may not be a minimal complex. We analyze separately the two above summands and show that any direct summand of them appearing in a miminal complex $F_{\beta}$ must be of the form $B_v(i)$ where $\ell(p(\beta))-\ell(v)$ and $i$ have the same parity. We have $$R(1)\otimes {^{i-1} F_{\beta'}}\cong R(1)\otimes \left({\bigoplus B_z(i-1)^{\oplus m_{z,i-1}}}\right)\cong {\bigoplus B_z(i)^{\oplus m_{z,i-1}}}.$$

Now if $\ell(z)-\ell(p(\beta))$ and $i$ have different parity, since $\ell(p(\beta))=\ell(p(\beta'))+1$ we get that $\ell(z)-\ell(p(\beta'))$ and $i-1$ also have different parity, hence that $m_{z,i-1}=0$, which concludes. 

We have $$B_s\otimes {^i F_{\beta'}}\cong B_s\otimes \left(\bigoplus B_z(i)^{\oplus m_{z,i}}\right)\cong\bigoplus B_s B_z(i)^{\oplus m_{z,i}}.$$

In this situation let $z$ such that $m_{z,i}\neq 0$ If $sz<z$, then $B_s B_z(i)\cong B_z(i+1)\oplus B_z(i-1)$. But by Theorem \ref{thm:linearity}, these summands cannot survive in a minimal complex for $F_\beta$. Hence assume that $sz>z$. Then 
$$B_s B_z(i)\cong B_{sz}(i)\oplus \bigoplus_{v<z} B_v(i) ^{\oplus \mu(v,z)},$$
Where $\mu(v,z)\in\mathbb{Z}_{\geq 0}$ and $\mu(v,z)\neq 0$ implies that $z$ and $v$ have distinct parity (see \cite[Section 7.11]{Humph}).
Hence assuming that $\ell(p(\beta))-\ell(v)$ (or $\ell(p(\beta))-\ell(sz)$) and $i$ have different parity, it follows that $\ell(p(\beta'))-\ell(z)$ have distinct parity, hence that $m_{z,i}=0$. It follows that the summands of $B_s B_z(i)$ are all of the form $B_v(i)$ where $\ell(p(\beta))-\ell(v)$ and $i$ have the same parity. This concludes the proof. 

\end{proof}

\subsection{Conjectures}

As mentioned in the Introduction, the generalized version of the positivity of the inverse Kazhdan-Lusztig polynomials involving the bases $\{T_{x,A}\}_{x\in\W}$ is the following statement \cite{Dyer:private}

\begin{conjecture}[Dyer]\label{conj:inverseposgen}
Let $x\in\W$, $A\in\Bic$. Then $$T_{x,A}\in\sum_{w\in\W} \mathbb{Z}_{\geq 0}[v^{\pm 1}] C_w.$$
\end{conjecture}
In Subsection \ref{sub:inverse}, we prove it in case $A=N(y)$ or $A=\T\backslash N(y)$ for $y\in\W$ (Theorem \ref{thm:inversepos}, see also Lemma \ref{lem:Mikado}), implying in particular that Property (\ref{p1}) holds for arbitrary Coxeter groups. In that case, the result followed from the linearity of the Rouquier complexes associated to left and right Mikado braids (Theorem \ref{thm:linearity}). We expect this last property to hold in general for Mikado braids:

\begin{conjecture}\label{thm:linearitygen}
Let $w\in \W$, $A\in\Bic$. Set $\beta:=w_A$. Let $F_\beta$ denote a minimal Rouquier complex for $\beta$. Then 

\begin{enumerate}
\item $^0 F_{\beta}$ has a unique summand isomorphic to $B_w$. All other summands of $^0 F_\beta$ are isomorphic to $B_z$ for $z<w$. 
\item For any $i\in\mathbb{Z}\backslash\{0\}$, we have $^i F_\beta=\bigoplus B_z(i)^{\oplus m_{z,i}}$ for $z<w$ and $m_{z,i}\in\mathbb{Z}_{\geq 0}$. In particular, we have $F_\beta\in\Kbn\cap\Kbp$.
\item We have $m_{z,i}=0$ if $i$ and $\ell(z)-\ell(w)$ have different parity.
\end{enumerate}
\end{conjecture}
If one proves $(1)$ and $(2)$, one can then argue exactly as in Subsection \ref{sub:inverse} to prove $(3)$ and derive Conjecture \ref{conj:inverseposgen}. 

%%%%%%%%%%%%%%%%%%%%%%%%%%%%%%%%%%%%%%%%%%%%%%%%%%%%%%%%%%%%%%%%%%%%%%%%%%%%%%%%%%%%

\end{document}